\documentclass[12pt,a4paper]{article}
\usepackage{amsfonts}
\usepackage{mathrsfs}

\usepackage{amssymb}
\usepackage{amsmath}
\usepackage{amscd}
\usepackage{amsthm}

\newtheorem{Thm}{Theorem}[section]
\newtheorem{Cor}[Thm]{Corollary}
\newtheorem{Lem}[Thm]{Lemma}
\newtheorem{Prop}[Thm]{Proposition}
\newtheorem{Def}[Thm]{Definition}

\newtheorem{Rem}[Thm]{Remark}
\theoremstyle{remark}

\newcommand{\mfk}{\mathfrak}
\newcommand{\mfkg}{\mathfrak{g}}
\newcommand{\mb}{\mathbb}
\newcommand{\mc}{\mathcal}

\newcommand{\Ld}{\Lambda}
\newcommand{\mfkhp}{\mfk{h}^+(\mathscr{A})}
\newcommand{\mfkhm}{\mfk{h}^-(\mathscr{A})}
\newcommand{\mfkcp}{\mfk{c}^+(\mathscr{A})}
\newcommand{\mfkcm}{\mfk{c}^-(\mathscr{A})}
\newcommand{\mfkhpm}{\mfk{h}^{\pm}(\mathscr{A})}

\newcommand{\mchp}{\mc{H}^+(\mathscr{A})}
\newcommand{\mchm}{\mc{H}^-(\mathscr{A})}
\newcommand{\mcd}{\mc{D}(\mathscr{A})}
\newcommand{\mcc}{\mc{C}(\mathscr{A})}
\newcommand{\mccp}{\mc{C}^+(\mathscr{A})}
\newcommand{\mccm}{\mc{C}^-(\mathscr{A})}
\newcommand{\mcf}{\mc{F}}
\newcommand{\mcp}{\mc{P}}
\newcommand{\mci}{\mc{I}}
\newcommand{\mcj}{\mc{J}}
\newcommand{\mcire}{\mc{I}^{re}}
\newcommand{\mciim}{\mc{I}^{im}}
\newcommand{\mbc}{\mb{C}}
\newcommand{\mbn}{\mb{N}}
\newcommand{\mbz}{\mb{Z}}
\newcommand{\lra}{\longrightarrow}
\newcommand{\ma}{\mathscr{A}}
\newcommand{\ges}{\geqslant}

\begin{document}

\begin{center}

{\Large \bf The double Ringel-Hall algebra on a hereditary abelian
finitary length category}

\bigskip

{\large {\bf Rujing Dou, Qunhua Liu, Jie Xiao}}
\bigskip

\end{center}

\begin{center}{ {Department of Mathematical Sciences, Tsinghua
University, Beijing 100084, P.R.China.}

{Email address: drj05@mails.tsinghua.edu.cn,
lqh98@mails.tsinghua.edu.cn, jxiao@math.tsinghua.edu.cn} }
\end{center}

\abstract{In this paper, we study the category
$\mathscr{H}^{(\rho)}$ of semi-stable coherent sheaves of a fixed
slope $\rho$ over a weighted projective curve. This category has
nice properties: it is a hereditary abelian finitary length
category. We will define the Ringel-Hall algebra of
$\mathscr{H}^{(\rho)}$ and relate it to generalized Kac-Moody Lie
algebras. Finally we obtain the Kac type theorem to describe the
indecomposable objects in this category, i.e. the indecomposable
semi-stable sheaves.}

\section{Introduction}
%Rather than $\mathscr{H}^{(\rho)}$, we consider the general case
%first: a hereditary abelian finitary length category, and construct
%the double Ringel-Hall algebra of this category.

It is known that by Ringel \cite{R1} and Green \cite{G}, the
composition  subalgebra of the Ringel-Hall algebra of a finite
dimensional hereditary algebra $\Lambda$ over a finite field $k$ is
isomorphic to the positive part of the quantized enveloping algebra
of the Kac-Moody Lie algebra $\mfkg$ determined by the Euler form.
By Sevenhant and Van den Bergh \cite{SV}, and Deng and Xiao
\cite{DX1}, the double Ringel-Hall algebra with a new gradation is
essentially isomorphic to the quantized enveloping algebra of a
generalized Kac-Moody algebra $\mfkg',$ which admits imaginary
simple roots and contains $\mfkg$ as a Lie subalgebra.  By Kac
\cite{K}, the set of dimension vectors of the indecomposable modules
of $\Lambda$ coincides with the set of positive roots of $\mfkg$.
Moreover, if $\alpha$ is a real root, there is a unique, up to
isomorphism, indecomposable module with dimension vector $\alpha$.
In \cite{DX2}, it was shown that the Ringel-Hall algebra approach
provides a new and self-contained proof of the above Kac theorem.
The case for valued quivers with loops was considered in \cite{WX}.

We generalize these results from the module category mod$(\Lambda)$
to an arbitrary hereditary abelian finitary length $k$-category
$\ma$. Recall that a category $\ma$ is called finitary if Hom-spaces
and Ext-spaces are finite sets. It is called a length category if
any object has a finite Jordan-H\"{o}lder series. To the category
$\ma$ we associate the composition Ringel-Hall algebra $\mcc$ and
the double Ringel-Hall algebra $\mcd$. By the Euler form on $\ma$ we
define two Borcherds-Cartan matrices $C_0$ and $C$, and associate to
them two generalized Kac-Moody Lie algebras $\mfkg(C_0)$ and
$\mfkg(C)$.

We prove that the double composition algebra $\mcc$ of the
Ringel-Hall algebra is isomorphic to the quantized enveloping
algebra of $\mfkg(C_0)$ (see Theorem \ref{compositiontheorem}); and
the double Ringel-Hall algebra $\mcd$ essentially gives a
realization of the quantized enveloping algebra of $\mfkg(C)$ (see
Theorem \ref{doubletheorem}). Roughly speaking, by enlarging the
Cartan part of $\mcd$ and giving a new gradation we obtain the
quantized enveloping algebra of $\mfkg(C)$. Furthermore, the set of
dimension vectors of the indecomposable objects in $\ma$ is the
union of the positive roots of $\mfkg(C_0)$ and $W_0(\cup_{s\geq
2}s\mci^{im})$, where $W_0$ is the Weyl group of $\mfkg(C_0)$ and
$\mci^{im}$ is the set of imaginary
simple roots of $\mfkg(C_0)$ %the subset which contains all simple
%imaginary roots and closed under the Wely group action
(see Theorem \ref{maintheoremtwo}). Moreover, if $\alpha$ is a
real root, there is a unique, up to isomorphism, indecomposable
object in $\ma$ with dimension vector $\alpha$.

As a special case, we consider the category $\mathscr{H}^{(\rho)}$
of semi-stable coherent sheaves of a fixed slope $\rho$ over a
weighted projective curve. All the above results can apply to this
category.

The paper is organized as follows. In Section 2, we recall the basic
knowledge of generalized Kac-Moody Lie algebra and its quantized
enveloping algebra. In Section 3 we define the composition and the
double Ringel-Hall algebra of a hereditary abelian
finitary length category $\ma$. %, whose symmetric Euler form is a
%so-called generalized Kac-Moody bilinear form. Then we can get the
%generalized Kac-Moody Lie algebra and its quantized enveloping
%algebra by the generalized Kac-Moody bilinear form.
The structure of the composition and the double Ringel-Hall algebra,
and the relation with generalized Kac-Moody Lie algebras are studied
in Section 4. %, and we also classify the dimension vectors of
%indecomposable objects.
In Section 5 we apply our results to the category
$\mathscr{H}^{(\rho)}$ and classify the dimension vectors of the
indecomposable semi-stable sheaves of slope $\rho$.
%Section 6 is the special case of the category $\mathscr{H}^{(\rho)}$ %of semi-stable coherent sheaves of
%a fixed slope $\rho$ over a weighted projective curve.
%We recall the main notions and results of the category of coherent
%sheaves over a weighted projective curve, and apply our main
%results to the category $\mathscr{H}^{(\rho)}$.

%The composition algebra is isomorphic to a generalized Kac-Moody
%Lie algebra which is determined by the Euler form . The double
%Ringel-Hall algebra is almost a larger quantized enveloping algebra (see
%Theorem \ref{doubletheorem}). Actually by extending the torus of
%the double Ringel-Hall algebra, we will get the quantized enveloping
%algebra. The set of dimension vectors of indecomposable objects is
%the union of the root system of the generalized Kac-Moody Lie
%algebra and the subset which contains all simple imaginary roots
%and closed under the Wely group action (see Theorem
%\ref{maintheoremtwo}).

Throughout the paper $k$ will be a fixed finite field
$\mathbb{F}_q$, and $v=\sqrt{q}$ be a complex number (and not a root
of unity).

 Finally we note that a deeper relation between the indecomposable
 coherent sheaves over weighted projective lines and the root system
 of the loop algebras of Kac-Moody algebras was found by
 Crawley-Boevey in \cite{CB}.

\bigskip
\section{Generalized Kac-Moody Lie algebras}

In this section we recall the definition of generalized Kac-Moody
Lie algebra and its quantized enveloping algebra. Generalized
Kac-Moody Lie algebras were introduced by Borcherds~\cite{B}, and
their quantized version was defined by Kang~\cite{Ka}. For reference
one sees also~\cite{SV} and~\cite{DX2}.

Let $\mci$ be an index set (possibly infinite or even
uncountable).

\begin{Def} A complex matrix $C=(c_{ij})_{i,j\in\mci}$ is called a
{\em Borcherds-Cartan matrix} if the following holds:

$(i)$ $c_{ii}=2$ or $c_{ii}\le 0$, for any $i\in\mci$;

$(ii)$ $c_{ij}\le0$, for any $i,j\in\mci$ and $i\ne j$;

$(iii)$ $c_{ij}\in\mb{Z}$, for any $i,j\in\mci$ and $c_{ii}=2$;

$(iv)$ $c_{ij}=0$ if and only if $c_{ji}=0$, for any $i,j\in\mci$.
\end{Def}

Set $\mci^{re}=\{i\in\mci: c_{ii}=2\}$ and
$\mci^{im}=\{i\in\mci:c_{ii}\le 0\}$. The index set $\mci$ is the
disjoint union of $\mcire$ and $\mciim$.

\begin{Def}
A Borcherds-Cartan matrix $C$ is called {\em symmetrizable}, if
there exists positive number $\varepsilon_i$ for $i\in\mci$
satisfying that $\varepsilon_i c_{ij}=\varepsilon_j c_{ji}$ for any
$i,j\in\mci$.
\end{Def}

Two symmetrizable Borcherds-Cartan matrices $C=(c_{ij})$ (with
symmetrization $\varepsilon_i$) and $C^\prime =(c_{ij}')$ (with
symmetrization $\varepsilon _i^\prime$) are identified, if they
correspond to the same symmetrization, namely $\varepsilon_i
c_{ij}=\varepsilon _i^\prime c _{ij}^\prime$ for any $i,j\in\mci$.

\begin{Rem} Under the above identification, a symmetrizable Borcherds-Cartan matrix is one-to-one
correspondent to a symmetric bilinear form
$(-,-):\mb{C}\mci\times\mbc\mci\lra\mbc$ satisfying that $(i,j)\le
0$ for any $i\ne j$ in $\mci$, and that if $(i,i)$ is positive then
$\frac{2(i,j)}{(i,i)}\in\mbz$. Such a bilinear form is called a {\em
generalized Kac-Moody bilinear form}. \label{matrixandbilinearform}
\end{Rem}

Indeed, given a symmetrizable Borcherds-Cartan matrix $C=(c_{ij})$
with symmetrization $\varepsilon_i$, the bilinear form defined by
$(i,j)=\varepsilon_ic_{ij}$, for any $i,j\in\mci$, is a
generalized Kac-Moody bilinear form. Conversely, one can associate
to a generalized Kac-Moody bilinear form
$(-,-):\mb{C}\mci\times\mbc\mci\lra\mbc$ a symmetrizable
Borcherds-Cartan matrix $C$ defined by $$c_{ij}=\begin{cases}
                     \frac{2(i,j)}{(i,i)}  &\text{if} (i,i)>0  \\
                     (i,j)  & \text{otherwise} \end{cases}$$
with symmetrization
$$\varepsilon_i=\begin{cases}
                   \frac{(i,i)}{2}  & \text{if}  (i,i)>0 \\
                   1 & \text{otherwise}.  \end{cases}$$
The pair $(\mci,(-,-))$ is called a {\em Borcherds datum}
following the notion Cartan datum of Lustig \cite{L}.

Recall that a complex matrix $C=(c_{ij})_{i,j\in\mci}$ is called a
{\em generalized Cartan matrix} provides that

$(i)$ $\mci$ is a finite set;

$(ii)$ $c_{ii}=2$, for any $i\in\mci$;

$(iii)$ $c_{ij}\in\mbz_{\le 0}$, for any $i, j\in\mci$ and $i\ne
j$;

$(iv)$ $c_{ij}=0$ if and only if $c_{ji}=0$, for any $i,j\in\mci$.

Clearly generalized Cartan matrices are Borcherds-Cartan matrices
with $\mci=\mcire$ being finite. In particular, symmetrizable
generalized Cartan matrices are symmetrizable Borcherds-Cartan
matrices.

\begin{Def}\label{generalized Kac-moody algebra} To a symmetrizable
Borcherds-Cartan matrix $C$, we associate a complex Lie algebra
$\mfk{g}(C)$, called the {\em generalized Kac-Moody Lie algebra},
which is generated by $\{e_i,f_i,h_i:i\in\mci\}$ with relations

$(i)$ $[h_i,h_j]=0$, $\forall\ i,j\in\mci$;

$(ii)$ $[h_i, e_j]=c_{ij}e_j$, $[h_i,f_j]=-c_{ij}f_j$, $\forall \
i,j\in\mci$;

$(iii)$ $[e_i,f_j]=\delta_{ij}h_i$, $\forall\ i,j\in\mci$;

$(iv)$ $(\text{ad}e_i)^{1-c_{ij}}e_i=0$,
$(\text{ad}f_i)^{1-c_{ij}}f_j=0$, $\forall\ i\in\mcire$ and
$j\in\mci$ with $i\ne j$;

$(v)$ $[e_i,e_j]=0$, $[f_i,f_j]=0$, $\forall\ i,j\in\mci$ with
$c_{ij}=0$.
\end{Def}

For $i\in\mcire$, we define a linear transformation
$\tilde{r}_i:\mbc\mci\lra\mbc\mci$ sending $j\in\mci$ to
$j-c_{ij}i$. The {\em Weyl Group} $W=W(C)$ of the generalized
Kac-Moody algebra $\mfk{g}(C)$ is the subgroup of $GL(\mbc\mci)$
generated by the reflections $\{\tilde{r}_i:i\in\mcire\}$. The
root system $\Delta=\Delta(C)$ of $\mfk{g}(C)$ can be described as
follows:
$$\Delta=\Delta_+\cup\Delta_-,\ \Delta_-=-\Delta_+,\
\Delta=\Delta^{re}\cup\Delta^{im},$$
$$\Delta^{re}=W(\mcire)=\{w(i):w\in W,\ i\in\mcire\},$$
$$\Delta^{im}=W(\mcf\cup-\mcf),$$
where $\mcf=\{0\ne \mu\in\mbn\mci:(\mu,i)\le 0,\forall\
i\in\mcire,\ \text{supp}(\mu) \text{ is
connected}\}\backslash\bigcup_{s\ge 2} s\mciim$, called the {\em
fundamental region}. Note that $i\in \mcire$ are real simple roots
and $i\in\mciim$ are imaginary simple roots.

\begin{Def} Let v be a complex number (not a root of unity). The
{\em quantized enveloping algebra} $U_v(\mfk{g}(C))$ of a
generalized Kac-Moody Lie algebra $\mfk{g}(C)$ is the algebra over
$\mbc$ generated by $\{E_i,F_i:i\in\mci\}$ and
$\{K_{\mu}:\mu\in\mbz\mci\}$ with relations
\label{quantizedgeneralizedKML}

$(i)$ $K_0=1,\ \ K_{\mu}K_{\nu}=K_{\mu+\nu},\ \ \forall\
\mu,\nu\in\mbz\mci;$

$(ii)$ $K_{\mu}E_i=v^{(\mu,i)}E_i K_{\mu},\ \
K_{\mu}F_i=v^{-(\mu,i)}F_i K_{\mu},\ \forall\ i\in\mci,\
\mu\in\mbz\mci;$

$(iii)$
$E_iF_j-F_jE_i=\delta_{ij}\frac{K_i-K_{-i}}{v_i-v_i^{-1}},\
\forall\ i,j\in\mci;$

$(iv)$ $\sum_{p=0}^{1-c_{ij}}(-1)^p \left[ \begin{array}{c} 1-c_{ij} \\
p\\ \end{array}\right]_{v_i} E_i^p E_j E_i^{1-c_{ij}-p}=0,\
\forall\ i\in\mcire,\ i\neq j\in\mci;$

$(iv)'$ $\sum_{p=0}^{1-c_{ij}}(-1)^p \left[ \begin{array}{c} 1-c_{ij} \\
p\\ \end{array}\right]_{v_i} F_i^p F_j F_i^{1-c_{ij}-p}=0,\
\forall\ i\in\mcire,\ i\neq j\in\mci;$

$(v)$ $E_iE_j-E_jE_i=0,\ F_iF_j-F_jF_i=0,\ \forall i,j\in\mci\
\text{with}\ c_{ij}=0;$ \\where $v_i=v^{\varepsilon_i}$
($\varepsilon_i$ are the symmetrization of $C$), and
$$[n]_{v_i}=\frac{v_i^n-v_i^{-n}}{v_i-v_i^{-1}},\ \ \ \ \ \ [n]_{v_i}!=\prod_{k=1}^n[k]_{v_i},$$
$$\left[\begin{array}{c}m\\n\end{array}\right]_{v_i}
=\frac{[m]_{v_i}!}{[m-n]_{v_i}![n]_{v_i}!}.$$
\end{Def}

The quantized enveloping algebra admits a natural triangle
decomposition $U_v(\mfk{g}(C))=U_v^-(\mfk{g}(C))\otimes
U_v^0(\mfk{g}(C))\otimes U_v^+(\mfk{g}(C))$, where the negative
part $U_v^-(\mfk{g}(C))$ is generated by $F_i$ and $K_{\mu}$, the
Cartan part by $K_{\mu}$, and the positive part by $E_i$ and
$K_{\mu}$. Define the formal character of $U_v^-(\mfk{g}(C))$ by
$$\text{ch}U_v^-(\mfk{g}(C))=\sum_{\mu\in\mbn
\mci}\text{dim}_\mathbb{C}U_v^-(\mfk{g}(C))_{-\mu}e(-\mu).$$
By~\cite{B} (see also~\cite{SV}), we have the following
proposition.

\begin{Prop}\label{formalcharacter}
The formal character of $U_v^-(\mfk{g}(C))$ is
$${\rm ch}U_v^-(\mfk{g}(C))=\prod_{\alpha\in\Delta^+}(1-e(-\alpha))^{-{\rm mult}_{\mfk{g}(C)}\alpha}.$$
\end{Prop}

%Let $(\mci,(-,-))$ be a Borcherds datum, and $C=(c_{ij})$ the
%corresponding Borcherds-cartan matrix with symmetrization
%$\varepsilon_i$ ($i\in\mci$). Let $\mfk{g}=\mfk{g}(C)$ be the
%associated generalized Kac-Moody Lie algebra and $U_v(\mfk{g})$
%the quantized enveloping algebra.
It is well-known that $U_v(\mfk{g}(C))$ is a Hopf algebra with
comultiplicaton $\Delta$, counit $\epsilon$ and the antipode $S$
given by
$$\Delta(E_i)=E_i\otimes 1+K_i\otimes E_i,\ \ \ \ \Delta(F_i)=F_i\otimes
K_{-i}+1\otimes F_i,\ \ \ \ \Delta(K_i)=K_i\otimes K_i,$$
$$\epsilon(E_i)=\epsilon(F_i)=0,\ \ \ \  \epsilon(K_i)=1,$$
$$S(E_i)=-K_{-i}E_i,\ \ \ \  S(F_i)=-F_iK_i,\ \ \ \ S(K_i)=K_{-i}.$$
Write $U_v^{\geq}(\mfk{g}(C))$ (respectively
$U_v^{\leq}(\mfk{g}(C))$) for the subalgebras of $U_v(\mfk{g})$
generated by $E_i, K_{\mu}$ (respectively by $F_i, K_{\mu}$). It
is clear they are Hopf subalgebras of $U_v(\mfk{g}(C))$.

Define a bilinear form $\phi:U_v^{\geq}(\mfk{g}(C))\times
U_v^{\leq}(\mfk{g}(C)) \lra \mbc$ by
$$\phi(E_i,F_j)=\delta_{ij}\frac{-1}{v_i-v_i^{-1}},\ \
\phi(K_i,K_j)=v^{-(i,j)},$$
$$\phi(K_i,F_j)=0=\phi(E_i,K_j),$$ for any $i,j\in\mci$, and
extend it according to the following relations: for any $a$, $a'$
in $U_v^{\geq}(\mfk{g}(C))$ and $b$, $b'$ in
$U_v^{\leq}(\mfk{g}(C))$,

$(i)$ $\phi(a,1)=\epsilon(a)$, $\phi(1,b)=\epsilon(b)$;

$(ii)$ $\phi(a,bb')=\phi(\Delta(a),b\otimes b')$;

$(iii)$ $\phi(aa',b)=\phi(a\otimes a',\Delta^{op}(b))$;

$(iv)$ $\phi(S(a),b)=\phi(a,S^{-1}(b))$,\\
where $\phi(a\otimes a',b\otimes b')=\phi(a,b)\phi(a',b')$, and
$\Delta^{op}(b)=\sum b_{2}\otimes b_{1}$, if $\Delta(b)=\sum
b_{1}\otimes b_{2}$.

Such a bilinear form $\phi$ satisfying $(i)-(iv)$ is called a {\em
skew-Hopf pairing}. Note that sometimes the triple
$(U_v^{\geq}(\mfk{g}),U_v^{\leq}(\mfk{g}),\phi)$ is called a
skew-Hopf pairing. %and $U_v(\mfk{g})$ is the reduced Drinfeld double. %(to see
%this, one only needs check the relation $(iii)$ in Definition
%\ref{quantizedgeneralizedKML} which follows from by Lemma
%\ref{pmvarphi}).

\begin{Prop} (Proposition 2.4 \cite{SV}) The skew-Hopf pairing $(U_v^{\geq}(\mfk{g}),U_v^{\leq}(\mfk{g}),\phi)$
defined above is restricted non-degenerate, that means its
restricted form $\phi:U_v^{+}(\mfk{g}) \times
U_v^{-}(\mfk{g})\longrightarrow \mathbb{C}$ is non-degenerate.
\label{pairingquantizedKML}
\end{Prop}

\bigskip
\section{The double Ringel-Hall algebra}

Let $k=\mb{F}_q$ be a fixed finite field, and $\mathscr{A}$ be an
abelian category. Assume that $\ma$ is $k$-linear, Hom-finite and
Ext-finite. That is, for all objects $X$, $Y$ and $Z$ in $\ma$,
the sets $\text{Hom}(X,Y)$ and $\text{Ext}^1(X,Y)$ are finite
dimensional $k$-vector spaces and the composition
$\text{Hom}(X,Y)\times\text{Hom}(Y,Z)\lra\text{Hom}(X,Z)$ is
$k$-bilinear. Assume further that $\ma$ is hereditary, i.e.
$\text{Ext}^i(-,-)$ vanishes for all $i\ges 2$, and that $\ma$ is
a length category, i.e. all objects in $\ma$ have a composition
series of finite length.

Following Ringel~\cite{R}, Green~\cite{G} and Xiao~\cite{X}, we
associate to such a category $\ma$ a Hopf algebra, called the {\em
Ringel-Hall algebra}, and a doubled version of the Ringel-Hall
algebra, by using a skew-Hopf pairing.

\subsection{Ringel-Hall algebras and skew-Hopf pairings}

Let $\ma$ be an abelian category as above. Let $\mcp$ be the set of
isomorphism classes of objects in $\mathscr{A}$, $\mcp_{1}$ the
complement set of $\{0\}$ in $\mcp$, and $\mci$ the set of
isomorphism classes of simple objects in $\mathscr{A}$. For
$\alpha\in\mcp$, write $M_{\alpha}$ for a representative object of
$\alpha$ in $\ma$. In particular if $i\in\mci$, we write $S_i$ for a
simple object in $\ma$ corresponding to $i$. The Grothendieck group
of the category $\mathscr{A}$ is the free abelian group $\mbz \mci$
with basis $\mci$, as $\ma$ is a length category. For any object $M$
in $\ma,$  write $\underline{\text{dim}}M$ for the image of $M$ in
$\mathbb{Z}\mci$, called the {\em dimension vector} of $M,$ which is
given by the composition factors of $M,$ or equivalently, uniquely
determined by the rule:
$\underline{\text{dim}}L=$\underline{\text{dim}}M+$\underline{\text{dim}}N$
for any exact sequence in $\ma:$ $0\rightarrow M\rightarrow
L\rightarrow N\rightarrow 0.$

For $\alpha$, $\beta$ and $\gamma$ in $\mcp$, the {\em Ringel-Hall
number} $g^{\gamma}_{\alpha\beta}$ counts the number of subobjects
$X$ of $M_{\gamma}$ satisfying $X\cong M_{\beta}$ and
$M_{\gamma}/X\cong M_{\alpha}$. The {\em Euler form} $\langle
\alpha,\beta\rangle$
=$\text{dim}_{k}\text{Hom}_{\mathscr{A}}(M_{\alpha},M_{\beta})-
\text{dim}_{k}\text{Ext}_{\mathscr{A}}^{1}(M_{\alpha},M_{\beta})$,
and the {\em symmetric Euler form}
($\alpha,\beta$)=$\langle\alpha,\beta\rangle+\langle\beta,\alpha\rangle$.
Denote by $a_{\alpha}$ the cardinality of the automorphism group of
$M_{\alpha}$.

\begin{Lem}The symmetric Euler form $(-,-): \mathbb{Z}\mci \times \mbz\mci \lra \mbz$
is a generalized Kac-Moody bilinear form. \label{genKM}\end{Lem}

\begin{proof} By definition, one needs to check: $(i)$ For $i,j\in \mci$ and $i\neq j$, $(i,j)\leq
0$, and $(ii)$ If $(i,i)>0$, then $\frac{2(i,j)}{(i,i)}\in\mbz$.

$(i)$ is obvious. For $(ii)$, note that for $i,j\in\mci$,
$\text{End}_{\ma}(S_i)$ is a finite skew field (hence a field), and
$\text{Ext}^1 _{\ma}(S_i,S_j)$ has the natural structure as
$\text{End}_{\ma}(S_j)$-$\text{End}_{\ma}(S_i)$-bimodule. Suppose
$(i,i)>0$. Namely
$$2(\text{dim}_k\text{End}_{\ma}(S_i)-\text{dim}_{k}\text{Ext}_{\ma}^{1}(S_i,S_i))=$$$$\text{dim}_k\text{End}_{\ma}(S_i)
(1-\text{dim}_{\text{End}_{\ma}(S_i)}\text{Ext}_{\mathscr{A}}^{1}(S_i,S_i))>0$$
This implies $S_i$ has no self-extensions, and
$(i,i)=2\text{dim}_{k}\text{End}_{\ma}(S_i)$. Hence, when $i\neq j$,
\begin{eqnarray*} \frac{2(i,j)}{(i,i)}&=&
-\frac{\text{dim}_{k}\text{Ext}_{\ma}^{1}(S_i,S_j)+\text{dim}_{k}\text{Ext}_{\ma}^{1}(S_j,S_i)}{\text{dim}_k\text{End}_{\ma}(S_i)}\\
&=&-\text{dim}_{\text{End}_{\ma}(S_i)}\text{Ext}_{\mathscr{A}}^{1}(S_i,S_j)-\text{dim}_{\text{End}_{\ma}(S_i)}\text{Ext}_{\mathscr{A}}^{1}(S_j,S_i)\end{eqnarray*}
is an integer.

\end{proof}

By Remark \ref{matrixandbilinearform}, the symmetric Euler form
determines a symmetrizable Borcherds-Cartan matrix, denoted by
$C_0=(c_{ij})_{i,j\in\mci}$, where $c_{ij}=\frac{2(i,j)}{(i,i)}$
if $(i,i)>0$ and $(i,j)$ otherwise, with symmetrization
$\varepsilon_i=\frac{(i,i)}{2}=\langle i,i\rangle$ if $(i,i)>0$
and $1$ otherwise.

Write $v$ for the complex number $\sqrt{q}$, and $v_i$ for
$v^{\varepsilon_i}$ for $i\in\mci$.

\begin{Def} The `positive' Ringel-Hall algebra, denoted by $\mchp $,
is defined to be the Hopf algebra over $\mbc$ with basis
$\{K_{\mu}u_{\alpha}^+:\mu\in\mbz \mci,\alpha\in\mcp\}$ whose Hopf
structure is given by the following: \label{defpositive}

$(i)$ (multiplication and unit)
$$u_{\alpha}^+u_{\beta}^+=v^{\langle\alpha,\beta\rangle}\sum_{\gamma\in\mcp}
g^{\gamma}_{\alpha\beta} u_{\gamma}^+,\
K_{\mu}K_{\nu}=K_{\mu+\nu},$$
$$K_{\mu}u_{\alpha}^+=v^{(\mu,\alpha)}u_{\alpha}^+K_{\mu},\
1=u_0^+=K_0;$$

$(ii)$ (comultiplication and counit)
$$\Delta(u_{\gamma}^+)=\sum_{\alpha,\beta\in\mcp}v^{\langle\alpha,\beta\rangle}
\frac{a_{\alpha}a_{\beta}}{a_{\gamma}} g^{\gamma}_{\alpha\beta}
u_{\alpha}^+ K_{\beta}\otimes u_{\beta}^+,$$
$$\Delta(K_{\mu})=K_{\mu}\otimes K_{\mu},\
\epsilon(u_{\alpha}^+)=\delta_{\alpha,0},\ \epsilon(K_{\mu})=1;$$

$(iii)$ (antipode)$$S(K_{\mu})=K_{-\mu},$$
$$S(u_{\gamma}^+)=\delta_{\gamma0}+\sum_{m\ge 1} (-1)^m
\sum_{\pi\in\mcp,\gamma_1,\cdots,\gamma_m\in\mcp_1}
v^{2\sum_{i<j}\langle \gamma_i,\gamma_j\rangle}
\frac{a_{\gamma_1}\cdots a_{\gamma_m}}{a_{\gamma}}
g^{\gamma}_{\gamma_1\cdots\gamma_m}
g^{\pi}_{\gamma_1\cdots\gamma_m} K_{-\gamma}u_{\pi}^+.$$
\end{Def}

\begin{Def} The `negative' Ringel-Hall algebra, denoted by $\mchm$,
is defined to be the Hopf algebra over $\mbc$ with basis
$\{K_{\mu}u_{\alpha}^-:\mu\in\mbz \mci,\alpha\in\mcp\}$ whose Hopf
structure is given by the following: \label{defnegative}

$(i)$ (multiplication and unit)
$$u_{\alpha}^-u_{\beta}^-=v^{\langle\alpha,\beta\rangle}\sum_{\gamma\in\mcp}
g^{\gamma}_{\alpha\beta} u_{\gamma}^-,\
K_{\mu}K_{\nu}=K_{\mu+\nu},$$
$$K_{\mu}u_{\alpha}^-=v^{-(\mu,\alpha)}u_{\alpha}^-K_{\mu},\
1=u_0^+=K_0;$$

$(ii)$ (comultiplication and counit)
$$\Delta(u_{\gamma}^-)=\sum_{\alpha,\beta\in\mcp}v^{\langle\beta,\alpha\rangle}
\frac{a_{\alpha}a_{\beta}}{a_{\gamma}} g^{\gamma}_{\beta\alpha}
u_{\alpha}^- \otimes u_{\beta}^-K_{-\alpha},$$
$$\Delta(K_{\mu})=K_{\mu}\otimes K_{\mu},\
\epsilon(u_{\alpha}^-)=\delta_{\alpha,0},\ \epsilon(K_{\mu})=1;$$

$(iii)$ (antipode) $$S(K_{\mu})=K_{-\mu},$$
$$S(u_{\gamma}^-)=\delta_{\gamma0}+\sum_{m\ge 1} (-1)^m
\sum_{\pi\in\mcp,\gamma_1,\cdots,\gamma_m\in\mcp_1}
\frac{a_{\gamma_1}\cdots a_{\gamma_m}}{a_{\gamma}}
g^{\gamma}_{\gamma_1\cdots\gamma_m}
g^{\pi}_{\gamma_m\cdots\gamma_1} u_{\pi}^-K_{\gamma}.$$
\end{Def}

\begin{Rem} See Schiffmann's lecture note~\cite{S}, in a more general
setting, for the proof of the Hopf structure defined as above. The
multiplication of Ringel-Hall algebras was defined by
Ringel~\cite{R}, the bialgebra structure was defined by
Green~\cite{G}, and the antipode was found by Xiao~\cite{X}.
\end{Rem}

Following Ringel~\cite{R1}, we define a bilinear form $\varphi
:\mchp\times\mchm\lra\mbc$ by
$$\varphi(K_{\mu}u_{\alpha}^+,K_{\nu}u_{\beta}^-)=v^{-(\mu,\nu)-(\alpha,\nu)+(\mu,\beta)}\frac{1}{a_{\alpha}}\delta_{\alpha\beta}$$
for any $\mu$, $\nu$ in $\mathbb{Z}\mci$ and $\alpha$, $\beta$ in
$\mcp$. In particular for $i\in\mci$,
$$\varphi(u_i^+,u_i^-)=\frac{1}{a_i}=\begin{cases}
                     \frac{1}{q^{\varepsilon_i}-1}=\frac{1}{v_i^2-1}  &\text{if } (i,i)>0  \\
                    \frac{1}{q^{{\rm dim}_k{\rm End}(S_i)}-1} & \text{otherwise}. \end{cases}$$
Similar to Xiao \cite{X} Proposition 5.3, we have the following
lemma.

\begin{Lem} With the bilinear form $\varphi$ defined as above,
$(\mchp,\mchm,\varphi)$ is a skew-Hopf pairing.
\label{bilineaform}

%for any $a$, $a'$ in $\mchp$ and $b$, $b'$ in $\mchm$,
%the following holds:

%$(i)$ $\varphi(a,1)=\epsilon(a)$, $\varphi(1,b)=\epsilon(b)$;

%$(ii)$ $\varphi(a,bb')=\varphi(\Delta(a),b\otimes b')$;

%$(iii)$ $\varphi(aa',b)=\varphi(a\otimes a',\Delta^{op}(b))$;

%$(iv)$ $\varphi(S(a),b)=\varphi(a,S^{-1}(b))$,\\
%where $\varphi(a\otimes a',b\otimes b')=\varphi(a,b)\varphi(a',b')$,
%and $\Delta^{op}(b)=\sum _{2}\otimes b_{1}$, if $\Delta(b)=\sum
%b_{1}\otimes b_{2}$.
\end{Lem}

One needs to check $\varphi$ satisfies similar relations as
$(i)-(iv)$ stated before Proposition \ref{pairingquantizedKML}.
The proof is straightforward and hence omitted.

\subsection{The double Ringel-Hall algebra}

The {\em double Ringel-Hall algebra} of the category $\ma$ is
defined to be the reduced Drinfeld double of the skew-Hopf pairing
$(\mchp,\mchm,\varphi)$, denoted by $\mcd$. It is the quotient of
the Hopf algebra $\mchp\otimes\mchm$ factoring out the Hopf ideal
generated by $\{K_{\mu}\otimes K_{-\mu}-1\otimes
1:\mu\in\mathbb{Z}\mci\}$, with the Hopf structure inherited from
$\mchp\otimes\mchm$. It has a triangle decomposition of the form
$$\mcd=\mfkhm\otimes\mc{T}\otimes\mfkhp,$$ where $\mc{T}$ is the subalgebra of $\mcd$ generated by
$\{K_{\mu}:\mu\in\mbz \mci\}$, and $\mfkhp$ ($\mfkhm$) is the
subalgebra of $\mchp$ ($\mchm$) generated by
$\{u_{\alpha}^+:\alpha\in\mcp\}$ ($\{u_{\beta}^-:\beta\in\mcp\}$,
respectively).
\begin{Lem} \label{pmvarphi}In $\mcd$ we have for $i,j\in \mci$ that
$$u_i^+u_j^--u_j^-u_i^+=-\varphi(u_i^+,u_j^-)(K_i-K_{-i}).$$
\end{Lem}

\begin{proof}
\begin{eqnarray*}
\Delta^2(u_i^+)&=&u_i^+\otimes 1\otimes 1+K_i\otimes u_i^+\otimes
1+K_i\otimes K_i\otimes u_i^+,\\
\Delta^2(u_j^-)&=&1\otimes 1\otimes u_j^-+1\otimes u_j^-\otimes
K_{-j}+u_j^-\otimes K_{-j}\otimes K_{-j}. \end{eqnarray*} Hence
\begin{eqnarray*}
u_j^-u_i^+&=&(1\otimes u_j^-)(u_i^+\otimes 1)\\
&=&\varphi(u_i^+,S(u_j^-))\cdot 1\otimes K_{-j}\cdot
\varphi(1,K_{-j})+\varphi(K_i,S(1))\cdot u_i^+\otimes u_j^-\cdot
\varphi(1,K_{-j})+\\
&&\ \ \varphi(K_i,S(1))\cdot K_i\otimes
1\cdot \varphi(u_i^+,u_j^-)\\
&=&-1\otimes K_{-j}\cdot \varphi(u_i^+,u_j^-)+K_i\otimes 1\cdot
\varphi(u_i^+,u_j^-)+u_i^+\otimes u_j^-,\end{eqnarray*} where
$\varphi(u_i^+,S(u_j^-))=\varphi(u_i^+,-u_j^-K_j)
=-\varphi(\Delta(u_i^+),u_j^-\otimes K_j)=-\varphi(u_i^+,u_j^-).$ On
the other hand, $u_i^+u_j^-=(u_i^+\otimes 1)(1\otimes
u_j^-)=u_i^+\otimes u_j^-$. Hence in the double Ringel-Hall algebra
$\mcd$, $u_i^+u_j^--u_j^-u_i^+=-\varphi(u_i^+,u_j^-)(K_i-K_{-i})$.
\end{proof}
%Following from Lemma \ref{pmvarphi}, we have
%$$u_i^+u_j^--u_j^-u_i^+=-\varphi(u_i^+,u_j^-)(K_i-K_{-i})$$ for any $i,j\in
%\mci$.

By setting
$\text{deg}(u_{\alpha}^+)=\underline{\text{dim}}(M_{\alpha})$,
$\text{deg}(u_{\alpha}^-)=-\underline{\text{dim}}(M_{\alpha})$ and
$\text{deg}(K_{\mu})=0$, for $\alpha\in\mcp$ and
$\mu\in\mathbb{Z}\mci$, the Hopf algebras $\mchp$ and $\mchm$ become
$\mbn \mci$-graded and $-\mbn \mci$-graded, respectively. Hence the
double Ringel-Hall algebra $\mcd$ is $\mbz \mci$-graded. For any
$\mu\in\mbn \mci$, the homogeneous space $\mfkhpm_{\pm\mu}$ is a
finite dimensional $\mbc$-vector space with basis
$\{u_{\alpha}^{\pm}:\alpha\in\mcp_{\mu}\}$, where
$\mcp_{\mu}=\{\alpha\in\mcp:\underline{\text{dim}}(M_{\alpha})=\mu\}$.

The double Ringel-Hall algebra $\mcd$ has an important algebra
automorphism $\omega:\mcd\lra\mcd$ defined on generators by
$$\omega(u_{\alpha}^+)=u_{\alpha}^-,\
\omega(u_{\alpha}^-)=u_{\alpha}^+,\ \omega(K_{\mu})=K_{-\mu},$$
for all $\alpha\in\mcp$ and $\mu\in\mbz \mci$. It is easy to see
that the operator $\omega$ is an involution, i.e. $\omega^2=id$,
and that $\omega$ induces algebra isomorphisms
$\mfkhp\stackrel{\simeq}{\lra}\mfkhm$ and
$\mchp\stackrel{\simeq}{\lra}\mchm$.

\begin{Lem}\label{omega} $(i)$ The operator $\omega$ is a coalgebra anti-morphism of $\mcd$, i.e.
$\Delta\circ\omega=\omega\circ\Delta^{op}$.

$(ii)$ For any $x\in\mchp$ and $y\in\mchm$, we have
$\varphi(x,y)=\varphi(\omega(y),\omega(x))$.

$(iii)$ The relation between $\omega$ and the antipode $S$ is
$S\circ \omega = \omega \circ S^{-1}$. \end{Lem}

\begin{proof} $(i)$ Since $\omega$ and $\Delta$ are algebra morphisms, it suffices
to check for the algebra generators $K_{\mu}$ ($\mu\in\mbz \mci$)
and $u_{\gamma}^{\pm}$ ($\gamma\in\mcp$). We have that
\begin{eqnarray*}
\Delta\circ\omega(K_{\mu})&=&\Delta(K_{-\mu})=K_{-\mu}\otimes
K_{-\mu}=\omega\circ\Delta^{op}(K_{\mu}),\\
\Delta\circ\omega(u_{\gamma}^+)&=&\Delta(u_{\gamma}^-)=\sum_{\alpha,\beta\in\mcp}
v^{\langle\beta,\alpha\rangle}\frac{a_{\beta}a_{\alpha}}{a_{\gamma}}
g^{\gamma}_{\beta\alpha} u_{\alpha}^-\otimes
u_{\beta}^-K_{-\alpha},\\
\omega\circ\Delta^{op}(u_{\gamma}^+)&=&\omega(\sum_{\alpha,\beta\in\mcp}
v^{\langle\alpha,\beta\rangle}
\frac{a_{\alpha}a_{\beta}}{a_{\gamma}} g^{\gamma}_{\alpha\beta}
u_{\beta}^+\otimes u_{\alpha}^+K_{\alpha})\\
&=&\sum_{\alpha,\beta\in\mcp} v^{\langle\alpha,\beta\rangle}
\frac{a_{\alpha}a_{\beta}}{a_{\gamma}} g^{\gamma}_{\alpha\beta}
u_{\beta}^-\otimes u_{\alpha}^-K_{-\beta}. \end{eqnarray*} So
$\Delta\circ\omega(u_{\gamma}^+)=\omega\circ\Delta^{op}(u_{\gamma}^+)$.
Similarly
$\Delta\circ\omega(u_{\gamma}^-)=\omega\circ\Delta^{op}(u_{\gamma}^-)$
holds.

$(ii)$ By the linearity of $\omega$ and the bilinearity of
$\varphi$, it is sufficient to check for basis elements. Take
$x=K_{\mu}u_{\alpha}^+$ and $y=K_{\nu}u_{\beta}^-$. Then
\begin{eqnarray*}
\varphi(\omega(y),\omega(x))&=&\varphi(K_{-\nu}u_{\beta}^+,K_{-\mu}u_{\alpha}^-)
=v^{-(\nu,\mu)+(\beta,\mu)-(\nu,\alpha)}
\frac{1}{a_{\beta}}\delta_{\alpha\beta}\\
&=&\varphi(x,y).\end{eqnarray*}

$(iii)$ It suffices to check for $K_{\mu}$ ($\mu\in\mbz \mci$) and
$u_{\gamma}^{\pm}$ ($\gamma\in\mcp$). We have
$$ S \circ \omega (K_{\mu}) = S(K_{-\mu}) = K_{\mu} = \omega\circ
S^{-1} (K_{\mu}),$$
$$ S \circ \omega (u_{\lambda}^+) = S(u_{\lambda}^-) = \delta_{\gamma0}+
\sum_{m\ge 1} (-1)^m
\sum_{\pi\in\mcp,\gamma_1,\cdots,\gamma_m\in\mcp_1}
\frac{a_{\gamma_1}\cdots a_{\gamma_m}}{a_{\gamma}}
g^{\gamma}_{\gamma_1\cdots\gamma_m}
g^{\pi}_{\gamma_m\cdots\gamma_1} u_{\pi}^-K_{\gamma}$$ $$=\omega
\circ S^{-1}(u_{\lambda}^+),$$ and similar for $u_{\lambda}^-$.
\end{proof}

Consequently, the skew-Hopf pairing
$\varphi:\mchp\times\mchm\lra\mbc$ defined in the last subsection
gives rise to a Hopf pairing $\psi:\mchp\times\mchp\lra\mbc$
defined by $\psi(a,b)=\varphi(a,\omega(b))$. That is, for any $a$,
$a'$ and $b$, $b'$ in $\mchp$, the following holds:

$(i)'$ $\psi(a,1)=\epsilon(a)$, $\psi(1,b)=\epsilon(b)$;

$(ii)'$ $\psi(a,bb')=\psi(\Delta(a),b\otimes b')$;

$(iii)'$ $\psi(aa',b)=\psi(a\otimes a',\Delta(b))$;

$(iv)'$ $\psi(S(a),b)=\psi(a,S(b))$.\\
For any $\mu\in\mbn \mci$ and $\alpha,\beta\in\mcp_{\mu}$, we have
$$\psi(u_{\alpha}^+,u_{\beta}^+)=\varphi(u_{\alpha}^+,u_{\beta}^-)=\frac{1}{a_{\alpha}}\delta_{\alpha\beta}.$$
So the restriction of $\psi$ to $\mfkhp_{\mu}$, and hence to
$\mfkhp=\bigoplus_{\mu\in\mbn \mci}\mfkhp_{\mu}$, is a definite
positive symmetric bilinear form.

\begin{Rem} Note that the Ringel-Hall algebra and the bilinear form $\psi$ can
actually be defined over the rational field $\mathbb{Q}$. So it
makes sense to talk about the definite positivity.
\end{Rem}

\bigskip
\section{Main results}

In Section 4.1 we clarify the relation of the double Ringel-Hall
algebra and its composition subalgebra with generalized Kac-Moody
Lie
algebras.%, using the uniqueness of skew-Hopf pairings (Lemma
%\ref{uniqueness}).
 \ In Section 4.2 we classify the dimension vectors
of indecomposable objects in the category $\ma$, via the root system
of the generalized Kac-Moody Lie algebra corresponding to the double
composition  algebra.

\subsection{The double composition  algebras}

Recall that $\ma$ is a hereditary abelian finitary length category,
and $\mcd$ the double Ringel-Hall algebra with a triangle
decomposition $\mcd=\mfkhm\otimes\mc{T}\otimes\mfkhp$.

Let $\mcc$ be the subalgebra of $\mcd$ generated by
$\{u_i^{\pm}:i\in \mci\}$ and $\mc{T}$, called the {\em double
composition algebra}. It is a Hopf subalgebra and $\mbz\mci$-graded
as well as $\mcd$. It also admits a triangle decomposition
$\mcc=\mfkcm\otimes\mc{T}\otimes\mfkcp$, where $\mfkcp$ (and
$\mfkcm$) are the subalgebra of $\mcc$ generated by $u_i^+\ (i\in
\mci)$ (and $u_i^-\ (i\in \mci)$, respectively). Let $\mccp$ and
$\mccm$ be the intersection of $\mcc$ with $\mchp$ and $\mchm$
respectively. So the involution $\omega$ of $\mcd$ defined before
Lemma \ref{omega} restricts to an involution of $\mcc$, switching
$u_i^+$ and $u_i^-$. It also induces algebra isomorphisms
$\mfkcp\stackrel{\simeq}{\lra}\mfkcm$ and
$\mccp\stackrel{\simeq}{\lra}\mccm$. Therefore it is clear that the
restriction to $\mccp\times\mccm$ of the bilinear form
$\varphi:\mchp\times\mchm\lra\mbc$, defined before Lemma
\ref{bilineaform}, gives rise to another %restricted non-degenerate
skew-Hopf pairing $(\mccp,\mccm,\varphi)$.

Recall that in Section 3.1 we defined the symmetric Euler form
$(-,-):\mathbb{Z}\mci\times \mb{Z}\mci\lra\mb{Z}$, which is a
generalized Kac-Moody bilinear form by Lemma \ref{genKM}. We write
$C_0$ for the corresponding Borcherds-Cartan matrix. Let
$\mathfrak{g}(C_0)$ be the associated generalized Kac-Moody Lie
algebra and $U_v(\mfk{g}(C_0))$ the quantized enveloping algebra
with generators $E_i, F_i$ ($i\in\mci$) and $K_{\mu}$
($\mu\in\mathbb{Z}\mci$) (see Definition
\ref{quantizedgeneralizedKML}).

% and
%$(\mci,(-,-))$ for the Borcherds datum.

%\begin{Lem} The restricted non-degenerate skew-Hopf pairing $(\mccp,\mccm,\varphi)$
%belongs to the Borcherds datum $(\mci,(-,-))$.
%\end{Lem}

%\begin{proof} The conditions in Definition \ref{pairingBorcherds}
%are naturally satisfied with $\mccp$ generated by $u_i^+$ and
%$\mc{T}$, and $\mccm$ generated by $u_i^-$ and $\mc{T}$.
%\end{proof}

%Let $\mathfrak{g}(C_0)$ be the generalized Kac-Moody Lie algebra
%associated to the Borcherds-Cartan matrix $C_0$. %By Proposition
%\ref{pairingquantizedKML},
%The quantized enveloping algebra $U_v(\mfk{g}(C_0))$ is the reduced
%Drinfeld double of a restricted non-degenerate skew-Hopf pairing
%$(U_v^{\geq}(\mfk{g}(C_0)),U_v^{\leq}(\mfk{g}(C_0)),\phi)$.% belonging
%to the Borcherds datum $(\mci,(-,-))$. Therefore, by the uniqueness
%lemma 3.4,
%The two skew-Hopf pairings
%$(U_v^{\geq}(\mfk{g}(C_0)),U_v^{\leq}(\mfk{g}(C_0)),\phi)$ and
%$(\mccp,\mccm,\varphi)$ are canonically isomorphic. Consequently,
%their reduced Drinfeld doubles are isomorphic.

\begin{Thm} The map $\Phi:U_v(\mathfrak{g}(C_0))\lra \mcc$ from the quantized enveloping algebra to
the double composition algebra, defined by
$$\Phi(K_i) =K_i,\\\ \Phi(E_i)=u_i^+,\ \  \Phi(F_i) =\begin{cases} -v_iu_i^-&
i\in\mci^{re}\\
 \frac{v^{2{\rm dim}_k{\rm End}_{\ma}(S_i)}-1}{v_i^{-1}-v_i}u_i^- & i\in\mci^{im} \end{cases}$$ is a
Hopf algebra isomorphism. \label{compositiontheorem}
\end{Thm}

\begin{proof} We proceed by two steps. The aim of step one is to show that $\Phi$ is a
well-defined surjective Hopf algebra homomorphism. We claim that
$\Phi(E_i)$, $\Phi(F_i)$ ($i\in\mci$) and $\Phi(K_{\mu})$
($\mu\in\mathbb{Z}\mci$) satisfy the relations $(i)-(v)$ in
Definition \ref{quantizedgeneralizedKML}. Indeed $(i)$ and $(ii)$
follow from Definition \ref{defpositive} and \ref{defnegative}
directly, and $(iii)$ follows from Lemma \ref{pmvarphi}. The
remaining relations $(iv)$, $(iv)'$ and $(v)$ follow from the same
calculation as in Proposition 6.2, 6.3 in \cite{DX1}, using the
fact that $i \in\mci^{re}$ iff $i$ has no self-extensions. Hence
$\Phi$ is a well-defined surjective algebra homomorphism. Since it
is clear that $\Phi$ commutes with the comultiplications and the
antipodes on the generators $E_i$, $F_i$ and $K_i$, step one is
done.

It remains to show that $\Phi$ is injective. Both
$U_v(\mathfrak{g}(C_0))$ and $\mcc$ admit a triangle decomposition
and they are compatible with $\Phi$. Also the Cartan parts are
preserved by $\Phi$. Recall that we have skew-Hopf pairings $\phi:
U_v^{\geq}(\mfkg(C_0))\times U_v^{\leq}(\mfkg(C_0))\lra
\mathbb{C}$ and $\varphi:\mccp\times\mccm\lra\mathbb{C}$. It is
straightforward to check that they are compatible with $\Phi$,
namely $\phi(a,b)=\varphi(\Phi(a),\Phi(b))$, for all $a\in
U_v^{\geq}(\mfk{g}(C_0))$ and $b\in U_v^{\leq}(\mfk{g}(C_0))$. For
example we check now $\phi(E_i,F_j)=\varphi(\Phi(E_i),\Phi(F_j))$.
Indeed if $j$ is real, then $\varphi(\Phi(E_i),\Phi(F_j))
=\varphi(u_i^+, -v_ju_j^-) = -v_i\frac{\delta_{ij}}{v_i^2-1}=
\delta_{ij}\frac{-1}{v_i-v_i^{-1}}= \phi(E_i,F_j)$. If $j$ is
imaginary, then \begin{eqnarray*}
\varphi(\Phi(E_i),\Phi(F_j))&=&\varphi(u_i^+, \frac{v^{2{\rm dim}_k{\rm End}_{\ma}(S_j)}-1}{v_j^{-1}-v_j}u_j^-)\\
&=&\frac{v^{2{\rm dim}_k{\rm
End}_{\ma}(S_j)}-1}{v_j^{-1}-v_j}\frac{\delta_{ij}}{v^{2{\rm
dim}_k{\rm End}_{\ma}(S_j)}-1}\\ &=&
\delta_{ij}\frac{-1}{v_i-v_i^{-1}}\\ &=& \phi(E_i,F_j).
\end{eqnarray*}
Suppose now $u\in U_v^+(\mathfrak{g}(C_0))$ lies in the kernel of
$\Phi$. For any $v\in U_v^-(\mathfrak{g}(C_0))$, we have that
$\phi(u,v)=\varphi(\Phi(u),\Phi(v))=0$. But
$\phi:U_v^{+}(\mfkg(C_0))\times U_v^{-}(\mfkg(C_0))\lra
\mathbb{C}$ is non-degenerate by Proposition
\ref{pairingquantizedKML}. Hence $u$ must be zero. This completes
the proof.
\end{proof}

%We have just showed the two skew-Hopf pairings are canonically
%isomorphic.
%Note that with the new generators $\Phi(u_i^{\pm})$,
%$(U_v^{\geq}(\mfk{g}(C_0)),U_v^{\leq}(\mfk{g}(C_0)),\phi)$ is
%still a restricted non-degenerate skew-Hopf pairing belonging to
%the same Borcherds datum. By Corollary \ref{corollary} it suffices
%to check the compatibility of the two bilinear forms, namely
%$\varphi(u_i^+,u_j^-)=\phi(\Phi(u_i^+),\Phi(u_i^-))$, which follows
%from the definition of $\varphi$ and $\phi$ immediately.

\subsection{The double Ringel-Hall algebra as the quantized enveloping algebra of a generalized Kac-Moody algebra}

Our next step is to measure the difference between the double
composition algebra $\mcc$ and the double Ringel-Hall algebra
$\mcd$, and to approximate $\mcd$ with the quantized enveloping
algebra of a larger generalized Kac-Moody Lie algebra, following the
method of Sevenhant and Van den Bergh \cite{SV} (used also by Deng
and Xiao \cite{DX1}).

We define $\Xi_0^{\pm}$ to be $\mfk{h}_0^{\pm}=\mb{C}$, and define
$\Xi_{i}^{\pm}$ to be $\mfk{h}^{\pm}(\Ld)_{\pm i}$ with basis
$\{u_i^{\pm}\}$ (for $i\in\mci$). For $\theta\in\mb{N}\mci$,
$\theta\neq 0$ and $\theta\notin\mci$, define $\Xi_{\theta}^{\pm}$
to be the subalgebra of $\mfk{h}^{\pm}(\Ld)_{\pm\theta}$ generated
by
$$\sum_{\mu+\nu=\theta,\mu,\nu\neq\theta}\mfk{h}^{\pm}(\Ld)_{\pm\mu}\mfk{h}^{\pm}(\Ld)_{\pm\nu}.$$
Set
$$L_{\theta}^+=\{x\in\mfk{h}^+(\Ld)_{\theta}:\varphi(x,\Xi_{\theta}^-)=0\}=\{x\in\mfk{h}^+(\Ld)_{\theta}:\psi(x,\Xi_{\theta}^+)=0\},$$
$$L_{\theta}^-=\{y\in\mfk{h}^-(\Ld)_{-\theta}:\varphi(\Xi_{\theta}^+,y)=0\}=\{y\in\mfk{h}^-(\Ld)_{-\theta}:\psi(y,\Xi_{\theta}^-)=0\}.$$
It is easy to see that
$\mfkhpm_{\pm\theta}=\Xi_{\theta}^{\pm}\oplus L_{\theta}^{\pm}$ as
a $\mb{C}$-vector space, and that $\omega: \Xi_{\theta}^+
\stackrel{\simeq}{\lra} \Xi_{\theta}^-$, $\omega: L_{\theta}^+
\stackrel{\simeq}{\lra} L_{\theta}^-$. Similar to Lemma 3.1 in
\cite{SV}, we show that elements in $L_{\theta}^{\pm}$ are
primitive, as well as $u_i^{\pm}$.

\begin{Lem}\label{primitivelement} (i) For any $x\in L_{\theta}^+$ and $y\in L_{\theta}^-$,
we have that
$$\Delta(x)=x\otimes 1+K_{\theta}\otimes x,\ \ S(x)=-K_{-\theta}x,$$
$$\Delta(y)=1\otimes y+y\otimes K_{-\theta},\ \ S(y)=-yK_{\theta}.$$

(ii) For any $x\in L_{\theta}^+$ and $y\in L_{\theta}^-$, we have
that $$xy-yx=-\varphi(x,y)(K_{\theta}-K_{-\theta}).$$
\end{Lem}

\begin{proof} $(i)$ We only need to prove the formula for $x\in L_{\theta}^+$, and
then apply the involution $\omega$ to obtain the formula for $y\in
L_{\theta}^-$.

For $\theta\in\mbn \mci$, let us take a normal orthogonal basis
$\{x_{(\theta,p)}:1\le p\le \text{dim}_{\mbc}L_{\theta}^+\}$ of
$L_{\theta}^+$, with respect to the definite positive bilinear
form $\psi:\mfkhp_{\theta}\times\mfkhp_{\theta}\lra\mbc$ defined
in the end of the last subsection. Write $\mc{J}$ for the index
set $\{(\theta,p):\theta\in\mbn \mci,\ 1\le p\le
\text{dim}_{\mbc}L_{\theta}^+\}$. Then
$\{x_{(\theta,p)}:(\theta,p)\in\mc{J}\}$ is a normal orthogonal
basis of $\oplus_{\theta\in\mbn \mci} L_{\theta}^+$. We extend it
to a normal orthogonal basis
$\{x_{(\theta,p)}:(\theta,p)\in\mc{J}'\}$ of $\mfkhp$, where
$\mc{J}'$ stands for the index set $\{(\theta,p):\theta\in\mbn
\mci,\ 1\le p\le \text{dim}_{\mbc}\mfkhp_{\theta}\}$. In
particular
$\psi(x_{(\theta,p)},x_{(\theta',p')})=\delta_{\theta,\theta'}\delta_{p,p'}$.
Note that
$\{x_{(\theta,p)}:(\theta,p)\in\mc{J}'\backslash\mc{J}\}$ forms a
basis of $\oplus_{\theta\in\mbn \mci} \Xi_{\theta}^+$, and each
element $x_{(\theta,p)}$ is homogeneous of degree $\theta$. For
example when $\theta=0$, $\mfkhp_0$ is one dimensional and
$x_{(0,1)}=1$. When $\theta=i\in\mci$, $\mfkhp_{i}=\Xi_{i}^+$ is
one dimensional and $x_{(i,1)}=\sqrt{a_i}\cdot u_i^+$, where $a_i$
is the cardinality of the automorphism group of the simple object
$S_i$ in $\ma$.

Suppose the comultiplication $\Delta$ sends a basis element
$x_{(\theta,p)}$ to a linear combination of the form
$\sum_{(\theta_1,p_1),(\theta_2,p_2)\in\mc{J}'}
c_{(\theta_1,p_1),(\theta_2,p_2)}x_{(\theta_1,p_1)}K_{\theta_2}\otimes
x_{(\theta_2,p_2)}$ with complex coefficients. Note that
$c_{(\theta,p),(0,1)}=1=c_{(0,1),(\theta,p)}$. % since
%$x_{(\theta,p)}$ are homogeneous.
For any $(\tau_1,q_1),(\tau_2,q_2)\in \mc{J}'$, we have
\begin{eqnarray*}\psi(x_{(\theta,p)},x_{(\tau_1,q_1)}x_{(\tau_2,q_2)})&=&\psi(\Delta(x_{(\theta,p)}),x_{(\tau_1,q_1)}\otimes
x_{(\tau_2,q_2)})\\ &=&\sum%_{(\theta,p),(\theta',p')\in\mc{J}'}
c_{(\theta_1,p_1),(\theta_2,p_2)}
\psi(x_{(\theta_1,p_1)}K_{\theta_2},x_{(\tau_1,q_1)})\psi(x_{(\theta_2,p_2)},x_{(\tau_2,q_2)})\\
&=&c_{(\tau_1,q_1),(\tau_2,q_2)},\end{eqnarray*} because
$\psi(x_{(\theta_2,p_2)},x_{(\tau_2,q_2)})=\delta_{\theta_2,\tau_2}\delta_{p_2,q_2}$,
and
$\psi(x_{(\theta_1,p_1)}K_{\theta_2},x_{(\tau_1,q_1)})=\psi(x_{(\theta_1,p_1)}\otimes
K_{\theta_2},\Delta(x_{(\tau_1,q_1)}))=\psi(x_{(\theta_1,p_1)}\otimes
K_{\theta_2},x_{(\tau_1,q_1)}\otimes
1)=\delta_{\theta_1,\tau_1}\delta_{p_1,q_1}$.

Now we take $(\theta,p)$ from $\mc{J}$ so that $x_{(\theta,p)}$
belongs to $L_{\theta}^+$. Then the bilinear form
$\psi(x_{(\theta,p)},x_{(\tau_1,q_1)}x_{(\tau_2,q_2)})$ is nonzero
only when either $\tau_1=\theta$ and $\tau_2=0$, or $\tau_1=0$ and
$\tau_2=\theta$. In the first case
$x_{(\tau_1,p_1)}=x_{(\theta,p)}$ and $x_{(\tau_2,q_2)}=1$, and in
the other case $x_{(\tau_1,p_1)}=1$ and
$x_{(\tau_2,q_2)}x_{(\theta,p)}$. Hence
$\Delta(x_{(\theta,p)})=x_{(\theta,p)}\otimes 1+K_{\theta}\otimes
x_{(\theta,p)}$. The first formula follows since the
comultiplication $\Delta$ is linear.

By definition of the comultiplication and counit we have for $x\in
L_{\theta}^+$ that $x=(\epsilon\otimes
1)\circ\Delta(x)=\epsilon(x)+\epsilon(K_{\theta})x=\epsilon(x)+x$.
Hence $\epsilon(x)=0$. By definition of the antipode we have that
$0=(S\otimes
id)\circ\Delta(x)=S(x)+S(K_{\theta})x=S(x)+K_{-\theta}x$. Hence
$S(x)=-K_{-\theta}x$.

$(ii)$ follows from $(i)$ and similar calculation with Lemma
\ref{pmvarphi}.
\end{proof}

We enlarge the index set $\mci$ to $\mci'=\mci\cup \mathcal {J}$,
where $\mathcal {J}=\{(\theta ,p):\theta\in\mbn \mci,1\le p\le
\text{dim}_{\mb{C}}L_{\theta}^+\}$. There exists a natural linear
map $\varpi:\mb{Z}\mci'\lra\mb{Z}\mci$ defined by $\varpi(i)=i$ for
$i\in \mci$, and $\varpi(j)=\theta$ for $j=(\theta,p)\in \mathcal
{J}$. The generalized Kac-Moody bilinear form on $\mbz\mci$ can be
extended to a bilinear form on $\mbz\mci'$ by
$(i,j)'=(\varpi(i),\varpi(j))$ for $i,j\in\mc{I}'$. We are going to
show that it is again a generalized Kac-Moody bilinear form (with
similar argument as Proposition 3.2 in \cite{SV}), and hence
determines a Borcherds datum $(\mci',(-,-)')$.

Note that $\mci'$ is a subset of $\mcj'$ appeared in the proof
above. To each $i\in\mci'$, we have also associated a primitive
homogeneous element
$x_i\in\mfkhp$ with degree $\varpi(i)$. Moreover, %they are orthogonal with
%respect to the positive definite bilinear form
$\psi(x_i,x_j)=\delta_{ij}$.

\begin{Prop}\label{bilinearform} The following holds for any $i,j\in\mci'$,

$(i)$ if $i\neq j$, $(i,j)'\le 0$;

$(ii)$ if $j\in\mc{J}$, $(j,j)'\leq 0$;

$(iii)$ if $(i,i)'>0$, $\frac{2(i,j)'}{(i,i)'}\in\mbz$. \\In
particular, $(-,-)':\mb{Z}\mci'\times\mb{Z}\mci'\lra\mb{Z}$ is a
generalized Kac-Moody bilinear form.
\end{Prop}

\begin{proof} $(i)$ Take any $i,j\in\mci'$ and $i\ne j$. We have
\begin{eqnarray*}
\Delta(x_ix_j)&=&\Delta(x_i)\Delta(x_j)=(x_i\otimes
1+K_{\text{deg}x_i}\otimes x_i)(x_j\otimes
1+K_{\text{deg}x_j}\otimes x_j)\\
&=&x_ix_j\otimes 1+x_i K_{\text{deg}x_j}\otimes
x_j+K_{\text{deg}x_i}x_j\otimes
x_i+K_{\text{deg}x_i}K_{\text{deg}x_j}\otimes
x_ix_j.\end{eqnarray*} Hence
\begin{eqnarray*}
\psi(x_ix_j,x_ix_j)&=&\psi(\Delta(x_ix_j),x_i\otimes x_j)
=\psi(x_iK_{\text{deg}x_j},x_i)\psi(x_j,x_j)\\
&=&\psi(x_i\otimes K_{\text{deg}x_j},\Delta(x_i))=\psi(x_i\otimes
K_{\text{deg}x_j},x_i\otimes 1+K_{\text{deg}x_i}\otimes x_i)\\
&=&1,\end{eqnarray*} and similarly
\begin{eqnarray*}\psi(x_ix_j,x_jx_i)&=&\psi(\Delta(x_ix_j),x_j\otimes
x_i)
=\psi(K_{\text{deg}x_i}x_j,x_j)\psi(x_i,x_i)\\
&=&\psi(K_{\text{deg}x_i}\otimes x_j,\Delta(x_j))
=\psi(K_{\text{deg}x_i}\otimes x_j,x_j\otimes 1+K_{\text{deg}x_j}\otimes x_j)\\
&=&\psi(K_{\text{deg}x_i},K_{\text{deg}x_j})=v^{(i,j)'}.\end{eqnarray*}
Now for any $a,b\in\mathbb{R}$, %and $ax_ix_j+bx_jx_i\in\mfkhp_{\text{deg}x_i+\text{deg}x_j}$, we have
because $\psi$ is positive definitive, $$0\le
\psi(ax_ix_j+bx_jx_i,ax_ix_j+bx_jx_i)=a^2+2v^{(i,j)'}ab+b^2.$$ It
follows that $v^{(i,j)'}\le 1$, and therefore $(i,j)'\le 0$, as
$v=\sqrt{q}>1$.

$(ii)$ Suppose $j\in\mcj$. Then $(i,j)'\le 0$ for any $i\in \mci$
by $(i)$. It follows that $(i,j)'\le0$ for any $i\in\mci'$. In
particular $(i,i)'\le 0$.

$(iii)$ Suppose $(i,i)'>0$. Then $i\in \mci$ by $(ii)$. Write
$\varpi(j)=\sum_{k\in \mci}a_k k$ where $a_k\in\mbz$. Then
$\frac{2(i,j)'}{(i,i)'}=\sum_{k\in
\mci}a_k\frac{2(i,k)}{(i,i)}\in\mbz$, because
$\frac{2(i,k)}{(i,i)}\in\mbz$ in an integer by Proposition
\ref{genKM}.
\end{proof}

It follows from the definition of $L_{\theta}^+$ that $x_i$
($i\in\mci'$) generates $\mfkhp$. Dually $y_i=\omega(x_i)$
($i\in\mci'$) generates $\mfkhm$. Let us denoted by
$C=(c_{ij}')_{i,j\in\mci'}$ (with symmetrization $\varepsilon_i$)
the symmetrizable Borcherds-Cartan matrix corresponding to the
Borcherds datum $(\mci',(-,-)')$. Namely
$$c_{ij}'=\begin{cases}
                     \frac{2(i,j)'}{(i,i)'}  &\text{if} (i,i)'>0  \\
                     (i,j)'  & \text{otherwise} \end{cases}$$
with symmetrization
$$\varepsilon_i=\begin{cases}
                   \frac{(i,i)'}{2}  & \text{if}  (i,i)'>0 \\
                   1 & \text{otherwise}.  \end{cases}$$
It is clear that the Borcherds-Cartan matrix $C_0=(c_{ij})$ with
index set $\mci$, is a submatrix of $C$.

\begin{Lem} The
following relations hold in $\mcd$:
$$\sum_{p=0}^{1-c_{ij}'}(-1)^p \left[ \begin{array}{c} 1-c_{ij}' \\
p\\ \end{array}\right]_{v_i} x_i^p x_j x_i^{1-c_{ij}'-p}=0,\
\forall\ i\in\mci'^{re}=\mcire,\ j\in\mci'\ \text{with}\ i\ne j,$$
$$\sum_{p=0}^{1-c_{ij}'}(-1)^p \left[ \begin{array}{c} 1-c_{ij}' \\
p\\ \end{array}\right]_{v_i} y_i^p y_j y_i^{1-c_{ij}'-p}=0,\
\forall\ i\in\mci'^{re}=\mcire,\ j\in\mci'\ \text{with}\ i\ne j,$$
$$x_ix_j-x_jx_i=0=y_iy_j-y_jy_i,\ \forall i,j\in\mci'\ \text{with
}c_{ij}'=0,$$ where $v_i=v^{\varepsilon_i}$ for $i\in\mci'$.
\label{maintheorem}
\end{Lem}

\begin{proof} We have shown in Proposition \ref{bilinearform} that
$(-,-)'$ is a generalized Kac-Moody bilinear form. Now the Lemma
follows from the same calculation as Proposition 6.2 and 6.3 in
\cite{DX1}.
\end{proof}

Let $\mfkg(C)$ be the generalized Kac-Moody algebra associated to
$C$ and $U_v(\mfkg(C))$ the quantized enveloping algebra. Recall
that $U_v(\mfkg(C))$ is generated by $E_i,F_i$ ($i\in\mci'$) and
$K_{\mu}$ ($\mu\in\mbz\mci'$) with respect to the relations given
in Definition \ref{quantizedgeneralizedKML}.

\begin{Thm} The map $\Psi:U_v(\mfk{g}(C))\lra\mcd$
from the quantized enveloping algebra to the double Ringel-Hall
algebra, defined by
$$\Psi(E_i)=x_i,\ \Psi(F_i)=\frac{1}{v_i^{-1}-v_i}y_i,\ \forall\ i\in \mci',$$
$$\Psi(K_{\mu})=K_{\varpi(\mu)},\ \forall\ \mu\in\mbz\mci',$$ is a Hopf algebra
epimorphism. Moreover, the restriction of $\Psi$ to $\mfkhp$ and
$\mfkhm$ gives rise to algebra isomorphisms to $U_v^+(\mfk{g}(C))$
and $U_v^-(\mfk{g}(C))$ respectively. \label{doubletheorem}
\end{Thm}

\begin{proof} We use the same strategy as in Theorem \ref{compositiontheorem}.
Firstly we show that $\Psi(E_i)$, $\Psi(F_i)$ ($i\in\mci'$) and
$\Psi(K_{\mu})$ ($\mu\in\mathbb{Z}\mci')$ satisfy the relations in
Definition \ref{quantizedgeneralizedKML}. Indeed the relations
$(i)$ and $(ii)$ follow from Definition \ref{defpositive} and
\ref{defnegative} and the fact that $x_i$ and $y_i$ are
homogeneous. The relation $(iii)$ follows from Lemma
\ref{primitivelement} $(ii)$: for $i=(\theta,p)\in\mci'$ (hence
$\theta=\varpi(i)$),
\begin{eqnarray*}
\Psi(E_i)\Psi(F_j)-\Psi(F_j)\Psi(E_i)&=&\frac{1}{v_j^{-1}-v_j}(x_iy_j-y_jx_i)\\
&=&\frac{1}{v_j^{-1}-v_j} (-\varphi(x_i,y_j)
(K_{\theta}-K_{-\theta}))\\ &=& \frac{1}{v_j^{-1}-v_j}
(-\psi(x_i,x_j) (K_{\theta}-K_{-\theta}))\\ &=& \delta_{ij}
\frac{K_{\theta}-K_{-\theta}}{v_i-v_i^{-1}}. \end{eqnarray*} The
remaining relations follow from Lemma \ref{maintheorem}. Note that
$\Psi$ is compatible with the comultiplications and antipodes by
Lemma \ref{primitivelement} $(i)$. Hence $\Psi$ is a well-defined
surjective Hopf algebra homomorphism.

For injectivity of $\Psi$ on $U_v^{\pm}(\mathfrak{g}(C))$, as we
argued in the proof of Theorem \ref{compositiontheorem}, it
suffices to show that $\Psi$ is compatible with the skew-Hopf
pairings $\phi$ and $\varphi$ on $U_v(\mfk{g}(C))$ and $\mcd$
respectively. Namely $\phi(a,b)=\varphi(\Psi(a),\Psi(b))$, for all
$a\in U_v^{+}(\mfk{g}(C)), b\in U_v^{-}(\mfk{g}(C))$. Indeed
$\varphi(\Psi(E_i),\Psi(F_j))=\varphi(x_i,
\frac{1}{v_j^{-1}-v_j}y_j) = \frac{\delta_{ij}}{v_i^{-1}-v_i} =
\phi(E_i,F_j)$. Actually the kernel of $\Psi$ is contained in the
Cartan part. It is the ideal generated by \{$K_j-K_\theta:
j=(\theta,p) \in\mcj$\}.
\end{proof}

\subsection{Positive roots and indecomposable objects}

Recall that the fundamental region of a generalized Kac-Moody Lie
algebra $\mfkg$ is given by $\mc{F}=\{0\ne\mu\in\mbn \mci:(\mu,i)\le
0,\ \forall\ i\in \mci^{re},\text{supp}(\mu) \text{ is
connected}\}\backslash \cup_{s\geq 2}s\mci^{im}$.

\begin{Lem} For nonzero $\theta\in\mb{N}\mci$, the space $L_{\theta}^{\pm}$
is zero unless $\theta$ lies in the union of the fundamental
region $\mc{F}_0$ of the generalized Kac-Moody Lie algebra
$\mfkg(C_0) $and $\cup_{s\geq 2}s\mci^{im}$.
\label{fundamentalregion}
\end{Lem}

\begin{proof} Take any nonzero $\theta\in\mbn \mci$ such that the space
$L_{\theta}^+$ is nonzero. Then $\theta$ does not belong to
$\mci$, since $\Xi_i^+=\mfkhp_i$ for $i\in \mci$ and hence
$L_i=0$. Now by Lemma \ref{bilinearform} $(i)$ for any $i\in \mci$
the bilinear form $(\theta,i)\le 0$.

We assume the support of such a $\theta $ is not connected and write
$\text{supp}(\theta)=X_1\cup X_2$, where $X_1$ and $X_2$ are
disjoint nonempty subsets of $\mci$ and the bilinear form $(i,j)=0$
for all $i\in X_1$ and $j\in X_2$. This means that objects of $\ma$
with dimension vector supported in $X_1$ and in $X_2$ have no
non-trivial extensions with each other. Therefore any object $M$
with dimension vectors $\theta$ can be decomposed into $M=M_1\oplus
M_2$, such that the dimension vector of $M_i$ is supported in $X_i$,
and $M_1$ and $M_2$ has no nontrivial extensions. This implies that
the multiplication of the elements corresponding to $M_1$ and $M_2$
in the Ringel-Hall algebra $\mfkhp$ gives rise to a unique term
corresponding to $M$. We have thus deduced
$\Xi_{\theta}^+=\mfkhp_{\theta}$, and hence $L_{\theta}^+=0$, which
is a contradiction to the choice of $\theta$.
\end{proof}

To the Borcherds-Cartan matrix $C$ indexed by $\mci'$ obtained in
the last subsection, we associate the root system
$\Delta=\Delta(C)$, the simple reflections
$\tilde{r}_i:\mbc\mci'\lra\mbc\mci'$ for $i\in \mci'^{re}$, and
the Weyl group $W=W(C)$ of the generalized Kac-Moody Lie algebra
$\mfk{g}(C)$ as defined in Section 2. To the Borcherds-Cartan
submatrix $C_0$ indexed by $\mci$, we associate the root system
$\Delta_0=\Delta(C_0)$, the simple reflection $r_i:\mbc
\mci\lra\mbc \mci$ for $i\in \mci$ and the Weyl group
$W_0=W(C_0)$. It is clear that $\Delta_0$ is a subsystem of
$\Delta$, and $\mci'^{re}=\mci^{re}$.

\begin{Lem} $(i)$ For any $i\in \mci^{re}$,
$\varpi\circ\tilde{r}_i=r_i\circ\varpi:\mbc\mci'\lra\mbc \mci$.

$(ii)$ There exists a group isomorphism $W\lra W_0$ by sending
$\tilde{r}_i$ to $r_i$, for $i\in \mci^{re}$.

$(iii)$ The linear map $\varpi:\mbz\mci'\lra\mbz \mci$ sends the
root system $\Delta$ to the union of the root system $\Delta_0$
and $W_0(\cup_{s\geq 2}s\mci^{im})$. \label{previous}
\end{Lem}

\begin{proof} $(i)$ It suffices to check that
$\varpi\circ\tilde{r}_i(j)=r_i\circ\varpi(j)$ for any $i\in
\mci^{re}$ and $j\in\mci'$ by the linearity of $\varpi$ and the
reflections. Recall that if we write $C=(c_{ij})$ then
$c_{ij}=\frac{2(i,j)'}{(i,i)'}=\frac{2(\varpi(i),\varpi(j)}{(\varpi(i),\varpi(i))}$
for all $i,j\in\mci'$. Hence,
\begin{eqnarray*}
\varpi\circ\tilde{r}_i(j)&=&\varpi(j-c_{ij}i)=\varpi(j)-c_{ij}i,\\
r_i\circ\varpi(j)&=&\varpi(j)-\frac{2(i,\text{deg}f_j)}{(i,i)}i=\varpi(j)-c_{ij}i.\end{eqnarray*}

$(ii)$ Because the matrix $C_0$ is a submatrix of $C$ and
$\mci'^{re}=\mci^{re}$, the simple reflections $r_i$ satisfy the
same relations as $\tilde{r}_i$. Indeed, the relations are
$r_i^2=1$ and $(r_ir_j)^{m_{ij}}=1$ where $m_{ij}=2,3,4,6$ or
$\infty$ depending on $\frac{4(i,j)^2}{(i,i)(j,j)}=0,1,2,3$ or $>
3$ respectively (see \cite{B}).

$(iii)$ It is clear that $\varpi$ identifies $\mci'^{re}$ with
$\mci^{re}$. By Lemma \ref{fundamentalregion}, the map $\varpi$
sends the fundamental region of the Lie algebra $\mfkg(C)$ to the
fundamental region of $\mfkg(C_0)$ and $\cup_{s\geq 2}s\mci^{im}$.
Since the root system is generated by the simple real roots and
the fundamental region under the Weyl group reflections, the
statement follows from $(i)$ and $(ii)$.
\end{proof}

Although we cannot classify the indecomposable objects in the
category $\ma$, we give the following correspondence between their
dimension vectors and the positive roots of the generalized
Kac-Moody Lie algebra $\mfkg(C_0)$.

\begin{Thm} Let $\Phi^+$ be the set of dimension vectors of
indecomposable objects in $\ma$. Then $\Phi^+=\Delta_0^+\cup
W_0(\cup_{s\geq 2}s\mci^{im})$ as subsets of $\mbn\mci$, and for
any real root $\alpha\in\Delta_0^+$ there exists a unique, up to
isomorphism, indecomposable object with dimension $\alpha$.
\label{maintheoremtwo}
\end{Thm}

\begin{proof} For $\alpha\in\Phi^+$, write $I(\alpha,q)$ for the number
of indecomposable objects in $\ma$ over $\mb{F}_q$ with dimension
vector $\alpha$. The formal character of $\mfkhm$ equals
$$\text{ch}\mfkhm =\sum_{\mu\in\mbn \mci}\text{dim}_{\mbc}\mfkhm_{-\mu}e(-\mu)
=\prod_{\alpha\in\Phi^+}(1-e(-\alpha))^{-I(\alpha,q)}.$$ By
Corollary \ref{maintheorem} and Proposition \ref{formalcharacter},
we have that
\begin{eqnarray*}
\text{ch}\mfkhm&=&\varpi(\text{ch}U_v^-(\mfk{g}(C))=\varpi(\Pi_{\beta\in\Delta^+}(1-e(-\beta))^{-\text{mult}_{\mfk{g}(C)}\beta})\\
&=&\Pi_{\beta\in\Delta^+}(1-e(-\varpi(\beta)))^{-\text{mult}_{\mfk{g}(C)}\beta}.\end{eqnarray*}
Hence
$$\Phi^+=\{\varpi(\beta):\beta\in\Delta^+\}$$ and
for any $\alpha$ in $\Phi^+$,
$$I(\alpha,q)=\sum_{\beta\in\Delta^+,\varpi(\beta)=\alpha}\text{mult}_{\mfk{g}(C)}\beta.$$
If $\alpha$ is a real root, then
$I(\alpha,q)=\text{mult}_{\mfk{g}(C)}\alpha=1$.

%It is clear that $\Delta_0^+\subset \varpi(\Delta^+)=\Phi^+$, then
By Lemma \ref{previous} (iii), it remains to prove that for any
$s\geq 2$ and $i\in\mci^{im}$, $si\in \Phi^+$. In the case
$(i,i)<0$, it is clear that $L_{2i}^{\pm}\neq 0$, so there exists
$(2i,1)\in \mathcal {J}$. It follows that $ni+(2i,1)$ is in the
fundamental region of the Lie algebra $\mfkg(C)$. Therefore its
image under $\varpi$, i.e. $(n+2)i$, lies in $\Phi^+$ (for any
$n\geq 0$). In the case $(i,i)=0$, let $k'$ denotes the field ${\rm
End}(S,S)$. We claim that there exists an indecomposable object
$L_{n}$ which has both the length and the Loewy length $n$, each
composition factor is $S$, and ${\text{
Hom}}(S,L_{n})=k'={\text{Ext}}^1(S,L_{n})$. This would finish the
proof of the Theorem.

We prove the claim by induction. Without lose the generality, we can
assume $k'=k$. When $n=1$, $L_1=S$ is the only choice. When $n=2$, .
Because ${\rm Ext}^1(S,S)=k'$, there is uniquely a non-split short
exact sequence $0\longrightarrow S \longrightarrow
L_2\longrightarrow S\longrightarrow 0$ with $L_2$ indecomposable.
Moreover we apply ${\text{Hom}}(S,-)$ to the sequence and find that
${\rm Hom}(S,L_2)=k'={\rm Ext^1}(S,L_2)$. Assume now there exists
such a unique $L_{n-1}$ satisfying that ${\rm
Hom}(S,L_{n-1})=k'={\rm Ext}^1(S,L_{n-1})$. Then there exists
uniquely a non-split sequence $0\longrightarrow
L_{n-1}\longrightarrow E\longrightarrow S\longrightarrow 0$. Note
that the Loewy length of $E$ must be $n-1\leq l.length(E)\leq n$. If
$l.length(E)=n-1$, then $E=L_{n-1}\oplus S$, namely the sequence
$0\longrightarrow L_{n-1}\longrightarrow E\longrightarrow
S\longrightarrow 0$ splits, that is a contradiction. Hence
$l.length(E)=n$. Since the length of $E$ is also $n$, $E$ must be
indecomposable and we define $L_n=E$. Apply ${\rm Hom}(S,-)$ to the
sequence and find that ${\rm Hom}(S,L_{n})=k'={\rm Ext}^1(S,L_{n})$.
\end{proof}

\section{Coherent sheaves on a weighted projective curve}

In this section we apply our main results to a special case, that
is, the category of coherent sheaves over a weight projective
curve. To be more precise, we fix a slope consider semistable
sheaves of the fixed slope so that we get a hereditary abelian
finitary length category. Our method will give the classification
of the dimension vectors of the indecomposable objects in this
category.

\subsection{General features}

We introduce the category of coherent sheaves on a weighted
projective curve axiomatically, following \cite{LR}. All results
in this subsection can be found in \cite{GL} and \cite{LR}.

Recall that our base field $k$ is a finite field. By a {\em category
of coherent sheaves on a weighted projective curve}, we mean a
category $\mathscr{H}$ satisfying the following axioms $(H1)-(H6)$
(and objects in $\mathscr{H}$ are {\em coherent sheaves}):

$(H1)$ $\mathscr{H}$ is an abelian, $k$-linear category.

$(H2)$ $\mathscr{H}$ is small and $\rm{Hom}$-finite.

$(H3)$ $\mathscr{H}$ admits a self-equivalence $\tau$ satisfying
Serre duality, i.e. ${\rm DExt}^1(X,Y)={\rm Hom}(Y,\tau X)$.

$(H4)$ $\mathscr{H}$ is Noetherian, and not each object having
finite length.\\
It follows from $(H3)$ that the category $\mathscr{H}$ is
hereditary. Assume $\mathscr{H}$ satisfies $(H1)-(H4)$, then each
$X\in\mathscr{H}$ has the form $X=X_0\oplus X_+$, where $X_0\in
\mathscr{H}_0=\{X\in \mathscr{H}\mid X\ \text{has finite
length}\}$ and $X_+\in \mathscr{H}_+=\{X\in \mathscr{H}\mid X\
\text{has no simple subobjects}\}.$ Sheaves in $\mathscr{H}_+$ are
called {\em bundles}. The full subcategory $\mathscr{H}_0$ splits
into a direct sum of tubes $\mathscr{H}_0=\bigsqcup_{X \in
C}\mathscr{S}_X$, where $\mathscr{S}_X=\mathbb{Z}A_\infty/\tau^n$
for some $n\geq0$ and $C$ is the index set. The simple objects $S$
in $\mathscr{H}_0$ are of two types: $(i)$ $\tau S\cong S$ (called
{\em ordinary}); $(ii)$ $\tau S\ncong S$ (called {\em
exceptional}).

$(H5)$ There exists an additive function ${\rm rk}:\
\mathscr{H}\rightarrow \mathbb{Z}_{\geq0}$, called the {\em rank},
such that for any object $X$ the following holds,

\ \ \ \ \ (i) ${\rm rk}(X)=0 \Leftrightarrow X\in \mathscr{H}_0$

\ \ \ \ \ (ii) ${\rm rk}(\tau X)={\rm rk}(X)$

\ \ \ \ \ (iii) $\exists X \in \mathcal {H}_+$, such that
${\rm rk}(X)=1$.\\
Bundles of rack one are called {\em line bundles}.

$(H6)$ $\mathscr{H}_0$ has only finitely many exceptional simples.
Moreover, for each tube $\mathscr{S}_X$ and any line bundle $L$,
we have $\sum_{S:\ \text{simples in} \mathscr{S}_X} {\rm
dim}_k{\rm Hom}(L,S)=1$.

\begin{Rem} The
index set $C$ has the structure of a smooth projective curve. Let
$p:\ C\rightarrow \mathbb{N}$ be defined by $p(X)=\# \{simples\ in\
\mathscr{S}_X\}$. Then $(C, p)$ is the weighted projective curve
associated to $\mathscr{H}$. %The category $\mathscr{H}$ has a
%tilting object if and only if  the underlying smooth projective
%curve C=$\mathbb{P}^1(k)$.
\end{Rem}

By definition the function $p$ takes value $1$ at ordinary points.
We collect the values at those exceptional points into a set
$(p_1, p_2,\cdots,p_t)$, called the {\em weight sequence}. The
{\em average Euler form} is the bilinear form on the Grothendieck
group $K_0(\mathscr{H})$ defined by
$$
\langle\langle
[X],[Y]\rangle\rangle=\frac{1}{p}\sum_{j=0}^{p-1}\langle[\tau^jX],[Y]\rangle
$$
where $p=l.c.m(p_1,p_2,\cdots, p_t)$ and
$\langle[X],[Y]\rangle=\text{dim}_k\text{Hom}(X,Y)-\text{dim}_k
\text{Ext}^1(X,Y)$ is the ordinary Euler form.

Fix a line bundle $L_0$. We define the {\em degree} to be the
additive function ${\rm deg}:\
K_0(\mathscr{H})\rightarrow\frac{1}{p}\mathbb{Z}$ via $${\rm
deg}([X])=\langle\langle [L_0],[X]\rangle\rangle-\langle\langle
[L_0],[L_0]\rangle\rangle {\rm rk}(X).$$ For a non-zero object
$X\in\mathscr{H}$, define the {\em slope} of $X$ to be
$\mu_X$:=$\frac{{\rm deg}(X)}{{\rm rk}(X)}$. We call the rational
number $g_\mathscr{H}=1-\langle\langle [L_0],[L_0]\rangle\rangle$
the {\em orbifold genus} of $\mathscr{H}$, and $\mathcal
{X}_\mathscr{H}=2(1-g_\mathscr{H})$ the {\em orbifold Euler characteristic} of $\mathscr{H}$.
%Let $(p_1,p_2,\cdots,
%p_t)$ be the weights sequence for $\mathscr{H}$, and let $\mathcal
%{X}_0$ be its Euler characteristic of curve C. Then  $\mathcal
%{X}_\mathscr{H}=\mathcal {X}_0-\sum_{i=1}^t(1-\frac{1}{p_i})$.
%(see~\cite{GL})

Write $g_0$ for the ordinary genus of the underlying smooth
projective curve. So $g_0=0$ means the projective line, and
$g_0=1$ means the elliptic curve.

\begin{Rem} The orbifold Euler characteristic $\mathcal {X}_\mathscr{H}$ can be used to classify the
category $\mathscr{H}$. If $\mathcal {X}_\mathscr{H}>0$ (domestic
case), all possible cases are as follows:

$(i)$ $g_0=0,\ no\ weights$ \ \ \ \        $(ii)$ $g_0=0,$\
weights (p)

$(iii)$ $g_0=0,$\ weights (p,q)\ \ \ \       $(iv)$ $g_0=0,$\
weights (2,2,n)

$(v)$ $g_0=0,$\ weights (2,3,3)\ \ \ \      $(vi)$ $g_0=0,$\
weights (2,3,4)

$(vii)$ $g_0=0,$\ weights (2,3,5).

If $\mathcal {X}_\mathscr{H}=0$ (tubular case), all the possible
cases are:

$(i)$ $g_0=0,$\ weights (3,3,3)\ \ \ \         $(ii)$ $g_0=0,$\
weights (2,4,4)

$(iii)$ $g_0=0,$\ weights (2,3,6)\ \ \ \       $(iv)$ $g_0=0,$\
weights (2,2,2,2)

$(v)$ $g_0=1,\ no\ weights$.

All the other cases are wild and $\mathcal {X}_\mathscr{H}<0$.
\end{Rem}

A non-zero bundle $X\in \mathscr{H}_+$ is called {\em stable}
(respectively, {\em semi-stable}) of slope $\rho\in \mathbb{Q}$ if
$(i)$ $\mu X=\rho$; $(ii)$ if $X'\varsubsetneq X,\ \mu X'<\mu X\
(\text{respectively}, \mu X'\leq\mu X).$ By \cite{GL} and
\cite{LR}, if the orbifold Euler characteristic
$\mc{X}_{\mathscr{H}}\geq0$, each indecomposable bundle is
semi-stable. Moreover if $\mc{X}_{\mathscr{H}}>0$, each
indecomposable bundle is stable.

\begin{Prop}(Riemann-Roch)
For each $X,Y \in K_0(\mathscr{H})$, we have
\[\begin{matrix}
\langle \langle X,Y\rangle\rangle=(1-g_\mathscr{H}){\rm rk}(X){\rm
rk}(Y)+\begin{vmatrix} {\rm rk}(X) & {\rm rk}(Y)\\{\rm deg}(X) &
{\rm deg}(Y)
\end{vmatrix}.
\end{matrix}\]
\end{Prop}

\begin{Lem}\label{exactsubcat}
Let $\mathscr{A}$ be a hereditary abelian category. For each
$\mathbb{Z}$-linear form $\lambda :K_0(\mathscr{A})
\longrightarrow \mathbb{Z}$, the full subcategory
$\mathscr{A}(\lambda)$ of $\mathscr{A}$ controlled by $\lambda$,
consisting of all objects $X$ from $\mathscr{A}$ with $(i)$
$\lambda(X)=0$ $(ii)$ $\forall X'\subseteq X$,
$\lambda(X')\leq\lambda(X)$, is an exact extension-closed
subcategory of $\mathscr{A}$. In particular
$\mathscr{A}(\lambda)$is again a hereditary abelian category.
\end{Lem}

\begin{proof}
We have to show $\mathscr{A}(\lambda)$ is closed under taking
kernels, cokernels and extensions. Let us only prove for kernels.
Take any $X,Y\in\mathscr{A}(\lambda)$, $f:X\longrightarrow Y$, the
short exact sequence $0\longrightarrow \text{Ker}(f)\longrightarrow
X\longrightarrow \text{Im}(f)\longrightarrow 0$ implies that
$\lambda(X)=\lambda(\text{Ker}(f))+\lambda(\text{Im}(f))$. On the
other hand, $\text{Ker}(f)\subseteq X$ and $\text{Im}(f)\subseteq
Y$, so $\lambda(\text{Ker}(f))\leq 0$ and $\lambda(\text{Im}(f))\leq
0$. We have $\lambda(\text{Ker}(f))=0$. Subobjects of
$\text{Ker}(f)$ are subobjects of $X$, therefore $\text{Ker}(f)\in
\mathscr{A}(\lambda)$.
\end{proof}

For each $\rho\in \mathbb{Q}$, let $\mathscr{H}^{(\rho)}$ be the
full subcategory of $\mathscr{H}$ consisting of all semi-stable
bundles of slope $\rho$ (including the $0$ object).

\begin{Prop}\label{semistablesubcat}
$\mathscr{H}^{(\rho)}$ is an exact subcategory of $\mathscr{H}$,
closed under extensions. Furthermore, each $X\in
\mathscr{H}^{(\rho)}$ has finite length in $\mathscr{H}^{(\rho)}$,
and the simple objects in $\mathscr{H}^{(\rho)}$ are the stable
ones.
\end{Prop}

\begin{proof}
Write $\rho=\frac{d}{r}$ with $d$ and $r$ coprime to each other.
Consider the linear form $\lambda=r{\rm deg}-d {\rm rk}:
K_0(\mathscr{H})\longrightarrow \mathbb{Z}$. By definition, the
subcategory $\mathscr{H}(\lambda)$ of $\mathscr{H}$ controlled by
$\lambda$ is just $\mathscr{H}^{(\rho)}$. Hence it follows from
Lemma \ref{exactsubcat} that $\mathscr{H}^{(\rho)}$ is an exact
subcategory of $\mathscr{H}$, closed under extensions.

It is clear that the stable objects are exactly the simple ones in
the $\mathscr{H}^{(\rho)}$. For each semi-stable $X$, if it is not
stable, there exists a proper subobject $X'$ of $X$, which is
semi-stable and of the same slope $\rho$. We can continue this
process on the $X'$, which has smaller rank. Finally we get a
stable object $Y$, and a short exact sequence $0\longrightarrow
Y\longrightarrow X\longrightarrow X/Y\longrightarrow 0$.
Therefore, $X$ has finite length in $\mathscr{H}^{(\rho)}$, and
the length is bounded by the rank $\text{rk}(X)$.
\end{proof}

\begin{Rem} If $\mathcal {X}_\mathscr{H}> 0$,
$\mathscr{H}^{(\rho)}$ is a semi-simple category. If $\mathcal
{X}_\mathscr{H}= 0$, we have the category equivalences
$\mathscr{H}^{(\rho)}\cong\mathscr{H}^{\infty}=\mathscr{H}_0.$
\end{Rem}
%$(ii)$If $\mathcal {X}_\mathscr{H}> 0$, let $E$ be an object in
%$\mathscr{H}^{(\rho)}$, the composition serial is $0\subseteq E_0
%\subseteq E_1 \subseteq\cdots\subseteq E_l =E$, $E_i/E_{i-1}$ is the
%stable object of slope $\rho$ for $i=1,2,\cdots,l$.
%$$
%{\rm Ext}^1(E_i/E_{i-1},E_{i-1}/E_{i-2})={\rm DHom}(E_{i-1}/E_{i-2},
%\tau (E_i/E_{i-1})),
%$$
%where $\tau (E_i/E_{i-1})$ is the stable bundle of slope $\mu
%(E_i/E_{i-1})-\mathcal {X}_\mathscr{H}=\rho-\mathcal
%{X}_\mathscr{H}<\rho$.
%We have ${\rm Ext^1}(E_i/E_{i-1},E_{i-1}/E_{i-2})=0$ . Hence, $E$ is semi-simple.\\
%If $\mathcal {X}_\mathscr{H}= 0$,see~\cite{S1} and~\cite{GL}.
%\end{proof}

\subsection{The category $\mathscr{H}^{(\rho)}$}

From now on, we fix a rational number $\rho\in\mathbb{Q}$ and
consider the full subcategory $\mathscr{H}^{(\rho)}$ of
$\mathscr{H}$ which consists of all semi-stable bundles of slope
$\rho$. By Proposition \ref{semistablesubcat},
$\mathscr{H}^{(\rho)}$ is a hereditary abelian finitary length
category and the simples in $\mathscr{H}^{(\rho)}$ are the stable
sheaves of slope $\rho$. Namely each semi-stable sheaf in
$\mathscr{H}^{(\rho)}$ has a composition series of finite length,
and the composition factors are stable sheaves of the same slope
$\rho$. Let $I$ be the set of isomorphism classes of stable objects
in $\mathscr{H}^{(\rho)}$. Then the Grothendieck group
$K_0(\mathscr{H}^{(\rho)})$ is isomorphic to the free abelian group
$\mbz I$ generated by $I$. As in Section 3.1, we call the image in
$\mbz I$ of a semi-stable sheaf its {\em dimension vector}.

We aim to understand the indecomposable objects in
$\mathscr{H}^{(\rho)}$. The wild cases $\mathcal {X}_\mathscr{H}< 0$
are difficult and very little is known. However, we obtain the
classification of the dimension vectors of the indecomposable
objects with the help of Ringel-Hall algebras and generalized
Kac-Moody
Lie algebras. %, because $\mathscr{H}^{(\rho)}$ is just one example
%of the hereditary abelian finitary length category.

On the Grothendieck group $\mbz I$ are defined the Euler form
$\langle [X],[Y]\rangle$
=$\text{dim}_{k}\text{Hom}_{\mathscr{\mathscr{H}^{(\rho)}}}(X,Y)-
\text{dim}_{k}\text{Ext}_{\mathscr{\mathscr{H}^{(\rho)}}}^{1}(X,Y)$
and the symmetric Euler form $([X],[Y])=\langle
[X],[Y]\rangle+\langle [Y],[X]\rangle$, where $X$ and $Y$ are
objects in $\mathscr{H}^{(\rho)}$. By Proposition \ref{genKM} the
later is a generalized Kac-Moody bilinear form. By Remark
\ref{matrixandbilinearform}, it uniquely determines a
symmetrizable Borcherds-Cartan matrix, say $C_0$. Let $\mfkg(C_0)$
be the corresponding  generalized Kac-Moody Lie algebra (see
Definition \ref{generalized Kac-moody algebra}) and
$U_v(\mfkg(C_0))$ the corresponding quantized enveloping algebra
(see Definition \ref{quantizedgeneralizedKML}).

On the other hand, we can associate a Hopf algebra, the double
Ringel-Hall algebra $\mc{D}(\mathscr{H})$, to the category
$\mathscr{H}^{(\rho)}$ (see Section 3). As studied in Sections 5,
the composition subalgebra of $\mc{D}(\mathscr{H})$ is isomorphic to
the quantized enveloping algebra $U_v(\mfkg(C_0))$ (see Theorem
\ref{compositiontheorem}). Moreover, the double Ringel-Hall algebra
$\mc{D}(\mathscr{H})$ is a quotient of the quantized enveloping
algebra $U_v(\mfkg(C))$ of some larger generalized Kac-Moody algebra
$\mfkg(C)$ which is obtained by extending the torus of $\mfkg(C_0)$
(see Theorem \ref{doubletheorem}).

We write $\Phi^+$ for the set of dimension vectors of
indecomposable objects in $\mathscr{H}^{(\rho)}$. Let $\Delta_0^+$
and $W_0$ be respectively the set of positive roots and the Weyl
group of $\mfkg(C_0)$ (see Section 2 for definition). Let $I^{im}$
be the subset of $I$ containing of isomorphism classes of the
simple objects in $\mathscr{H}^{(\rho)}$ which have nontrivial
self-extensions. Indeed $I^{im}$ is the set of imaginary simple
roots of $\mfkg(C_0)$.

Finally, we have the following Corollary of Theorem
\ref{maintheoremtwo}, which relates the set $\Phi^+$ to the
positive roots of $\mfkg(C_0)$.

\begin{Cor} As subsets of $\mathbb{N}I$ we have
$\Phi^+ = \Delta_0^+ \cup W_0(\cup_{s\geq 2} sI^{im})$. Moreover
for each real root $\alpha \in \Delta_0^+$, there exits a unique
(up to isomorphism) indecomposable object with dimension vector
$\alpha$.
\end{Cor}

Hence, for a fixed slop $\rho$, the dimension vector of an
indecomposable semi-stable sheaf of slope $\rho$ is either a
positive root of the generalized Kac-Moody algebra $\mfkg(C_0)$,
or a imaginary root lying in $W_0(\cup_{s\geq 2} sI^{im})$. By
definition the positive roots of $\mfkg(C_0)$ is obtained by
applying the Weyl group $W_0$ to the real simple roots
$I^{re}=I\backslash I^{im}$. Hence the set $\Phi^+ = \Delta_0^+
\cup W_0(\cup_{s\geq 2} sI^{im})$ can be computed if the Euler
form is given.

\begin{Rem}
Suppose $\mathscr{H}$ is not weighted. Let $X,Y$ be the stable
bundles in $\mathscr{H}^{(\rho)}$, then the Euler form $\langle
X,Y\rangle=(1-g_0){\rm rk}(X){\rm rk}(Y)$.
\end{Rem}

\end{document}